\documentclass{amsart}
\usepackage[utf8]{inputenc}
\usepackage{graphicx}
\usepackage{tikz-cd}
\usepackage{mathrsfs}
\usepackage{amsfonts}
\usepackage{amsthm}
\usepackage{thmtools}
\usepackage{amsmath}
\usepackage{amssymb}
\usepackage{mathtools}
\usepackage{extarrows}
\usepackage{interval}
\usepackage{bbm}
\usepackage[english]{babel}
\usepackage{romannum}
\usepackage{setspace}
\usepackage{listings}
\usepackage{pdfpages}
\usepackage{hyperref}
\usepackage{enumitem}
\usepackage{cleveref}
\usepackage{verbatim}

\newtheorem{theorem}{Theorem}[section]
\newtheorem{lemma}[theorem]{Lemma}
\newtheorem{corollary}[theorem]{Corollary}
\newtheorem{proposition}[theorem]{Proposition}

\newtheorem{maintheorem}{Theorem}

\theoremstyle{definition}
\newtheorem{definition}[theorem]{Definition}

\newtheorem{example}[theorem]{Example}
\numberwithin{equation}{section}
\newcommand*{\sheafhom}{\mathcal{H}\kern -.5pt om}
\newcommand*{\sheafext}{\mathcal{E}\kern -.5pt xt}

\title[coarse moduli superspaces]{Existence of coarse moduli superspaces for Deligne--mumford superstacks}
\author{Fei Peng}
\email{pengf2@student.unimelb.edu.au}
\address{School of Mathematics \& Statistics, The University of Melbourne, Parkville, VIC, 3010, Australia}
\date{\today}

\begin{document}

\begin{abstract}
    We prove that every Deligne--Mumford superstack over $\mathbb{Z}[1/2]$ with finite inertia admits a coarse moduli superspace. This generalizes the Keel--Mori theorem to the super setting. As part of our argument, we introduce a notion of integral morphisms in algebraic supergeometry and use it to establish an effective descent result for \'etale morphisms of algebraic superspaces. We also include an example of a non-Deligne--Mumford superstack with finite inertia that does not admit a coarse moduli superspace.  
\end{abstract}

\maketitle
\pagenumbering{arabic}
\tableofcontents

\section{Introduction}

Supergeometry has received increasing attention in the past few decades, partly through the emergence of moduli problems in supersymmetry and its connections with other areas of mathematical physics. For example, the supermoduli of super Riemann surfaces (also known as supersymmetric curves) play a central role in algebraic supergeometry. They have been constructed both analytically and algebraically in the literature \cite{LR88, DHS97}. Recent work by Witten has added a modern perspective to this topic \cite{Wit19}. In particular, the theory of superstacks has proven to be highly effective for studying supermoduli spaces of stable Riemann surfaces. See \cite{Nor26} for an example. Codogni and Viviani constructed the supermoduli of super Riemann surfaces as smooth separated complex Deligne--Mumford superstacks \cite[Theorem\ B]{CV19}. More recently, Felder, Kazhdan, and Polishchuk obtained compactifications of the supermoduli of super Riemann surfaces with punctures as Deligne--Mumford superstacks \cite{FKP23} using stable super Riemann surfaces introduced by Deligne in his letter to Manin. This result was independently established by Moosavian and Zhou \cite{MZ24}.

The theory of superstacks, however, is still at an early stage of development. It first appeared in \cite{CV19}, where the authors introduced the notion of complex supergroupoids and complex Deligne--Mumford superstacks. The foundations of algebraic superstacks were established in full generality only recently in \cite{BH25}. More recently, methods of derived algebraic geometry have also been extended to the super case in \cite{Dan26}. It is natural to extend classical results of algebraic stacks to the super setting. One of the most important tools in moduli theory is the existence of moduli spaces for algebraic stacks. For Deligne--Mumford stacks, the Keel--Mori theorem \cite{KM97,Con05,Ryd13} provides a general mechanism for passing from a stack to an algebraic space that retains its topological and geometric properties. In the super setting, there is evidence of supermoduli problems admitting such moduli spaces. See \cite[Theorem\ B(2)]{CV19} for example. However, an analogous result does not seem to be available even for complex Deligne--Mumford superstacks.

The goal of this paper is to establish the existence of coarse moduli superspaces for Deligne--Mumford superstacks. More precisely, we prove the following result.

\begin{maintheorem}\label{thm:main}
    Let $\mathcal{X}$ be a Deligne--Mumford superstack over $\mathbb{Z}[1/2]$. If $\mathcal{X}$ has finite inertia, then it admits a coarse moduli superspace $\pi\colon\mathcal{X}\to X$ such that $\pi$ is a separated universal homeomorphism. Let $S$ be an algebraic superspace and let $\mathcal{X}\to S$ be a morphism of superstacks. If $\mathcal{X}\to S$ is locally of finite type, then $\pi$ is proper and quasi-finite. Consider the following properties
    \begin{enumerate}[topsep=0pt,noitemsep,label=\normalfont(\Alph*)]
        \item\label{theorem_A} quasi-compact, quasi-separated, separated, universally open, and universally closed,
        \item\label{theorem_B} locally of finite type, finite type, and proper.
    \end{enumerate}
    If $\mathcal{X}\to S$ has any of the properties in \ref{theorem_A}, so does $X\to S$. If $S$ is locally Noetherian and $\mathcal{X}\to S$ has any of the properties in \ref{theorem_B}, so does $X\to S$.
\end{maintheorem}

In particular, \Cref{thm:main} implies that any separated Deligne--Mumford superstack over $\mathbb{C}$ admits a coarse moduli superspace, confirming an expectation of Codogni and Viviani \cite[Page\ 370]{CV19}.  For example, it applies to the moduli superstack of super Riemann surfaces $\mathcal{SM}_g$ and its compactification $\overline{\mathcal{SM}}_g$. In particular, \Cref{thm:main} is used in forthcoming work with Ott and Polishchuk to show that $\overline{\mathcal{SM}}_g$ admits a projective coarse moduli superspace \cite{OPP26}. It also applies to many moduli superstacks of recent interest, such as the moduli superstack of branched covers by \cite[Theorem\ 21]{DO23} and the moduli superstack of stable supermaps by \cite{BH25b}. 

Our method for \Cref{thm:main} is inspired by the strategy of Rydh for \cite[Theorem\ 6.13]{Ryd13}. We proceed in the following steps.
\begin{enumerate}
    \item Produce a separated, \'etale, representable, stabilizer-preserving surjective morphism $\mathcal{W}\to\mathcal{X}$ where $\mathcal{W}$ admits a finite and \'etale presentation from an AF superscheme $V$. This is essentially \Cref{prop:stabilizer_preserving}, and it is the only step where the assumption on the inertia superstack is used.
    \item Show that $\mathcal{W}$ admits a coarse moduli superspace $\mathcal{W}\to W$ such that the composition $V\to\mathcal{W}\to W$ is integral. We will do so in \Cref{thm:superstacky_AF_mod_finite}.
    \item Show that $\mathcal{Q}=\mathcal{W}\times_{\mathcal{X}}\mathcal{W}$ admits coarse moduli superspaces $\mathcal{Q}\to Q$ such that the two projections $\mathcal{Q}\rightrightarrows\mathcal{W}$ descend to an \'etale equivalence relation $Q\rightrightarrows W$, and conclude that $\mathcal{X}$ admits a coarse moduli superspace $\mathcal{X}\to X$, where $X$ is the quotient of the \'etale equivalence relation $Q\rightrightarrows W$. This is explained in \Cref{thm:superstacky_Kollar_Rydh}. 
\end{enumerate}
The main difficulty is that several ingredients used in the ordinary theory are not currently available in the supergeometric setting. To overcome this, we introduce a theory of integral extensions of superrings and integral morphisms for superschemes, and we establish various descent results for étale morphisms in the super case. These results appear to be new in algebraic supergeometry and may be of independent interest.

In the ordinary case, the Keel--Mori theorem implies that every algebraic stack with finite inertia admits a coarse moduli space. However, this is false in the super case. See \Cref{ex:counter_Keel_Mori} for a counterexample. In particular, this counterexample reveals an obstruction arising from odd stabilizers.

Unless otherwise stated, we work over $\mathbb{Z}[1/2]$. The article is organized as follows. In \Cref{sec:basics}, we recall some basic facts in commutative superalgebra and supergeometry, including the definition of supergroupoids and superstacks. We also introduce various notions of quotients by supergroupoids. In \Cref{sec:descent}, we introduce integral extensions of superrings and integral morphisms of superschemes. We then show that integral surjective morphisms of superschemes are of effective descent for separated and \'etale morphisms, extending results in \cite{Ryd10} to the super setting. We also show that integral strongly geometric quotients satisfy the descent condition for separated and \'etale morphisms. In \Cref{sec:keel_mori}, we prove \Cref{thm:main}. We first show that the quotient stack of an AF superscheme by a finite and \'etale supergroupoid admits a coarse moduli superspace and then perform a reduction step by constructing a separated, \'etale, representable, surjective, and stabilizer-preserving morphism. In the two appendices, we explain the construction of bosonic quotients and representability of the Hilbert space of points in the super case.

\subsection*{Acknowledgement}

This project is partially supported by the Australian Research Council DP210103397 and FT210100405, a Melbourne Research Scholarship, and the Science Abroad Travelling Scholarship offered by the University of Melbourne. I would like to thank my advisor, Jack Hall, for his guidance, support, and encouragement throughout this project. It is my pleasure to thank Nadia Ott and Alexander Polishchuk for helpful discussions and comments on an earlier version of this manuscript. I would also like to thank Marcel Dang for many enlightening conversations on supergeometry. I am also grateful to Paul Norbury for discussions and advice on supergeometry. It is also my pleasure to thank Tianqi Feng, Oliver Li, and Adam Monteleone for their helpful comments, suggestions, and encouragement. 

\subsection*{AI disclosure}

This work was initiated in 2025 as part of the author's PhD thesis. The main ideas were developed by the author. Generative AI tools were subsequently used in the preparation of this article to help with proofreading and writing improvements. The author takes full responsibility for the correctness of all results in this paper. 

\section{Basics on algebraic supergeometry}\label{sec:basics}

In this section, we briefly recall some basic concepts in commutative superalgebra and algebraic supergeometry for the reader's benefit. For the details on commutative superalgebra, we recommend Westra's thesis \cite{Westra:2009}. For the foundations of algebraic supergeometry and superstacks, we generally follow \cite{BHP23} and \cite{CV19}. A reader familiar with the language of algebraic supergeometry can safely skip to the next section. Throughout this section, we work over $\mathbb{Z}[1/2]$.

\subsection{Commutative superalgebra}

 Let $R$ be a $\mathbb{Z}_{2}$-graded ring such that $R=R_{0}\oplus R_{1}$. We say an element $r\in R$ is odd (resp. even) if $r$ is contained in $R_{1}$ (resp. $R_{0}$). An element in $R$ is homogeneous if it is either odd or even. This induces a map from the set of homogeneous elements in $R$ to $\mathbb{Z}_2$, which we denote by $|-|$.

\begin{definition}\label{def:superrings}
    A \textit{superring} $A$ is a $\mathbb{Z}_{2}$-graded unital ring $A=A_{0}\oplus A_{1}$ such that $A_{i}A_{j}\subseteq A_{i+j}$ for all $i,j\in\{0,1\}$. Morphisms of superrings are simply ring homomorphisms that preserve $\mathbb{Z}_{2}$-gradings. A superring $A$ is \textit{supercommutative} if for every pair of homogeneous elements $a,b\in A$,$$ab=(-1)^{|a||b|}ba.$$
\end{definition}

It follows from the definition that if $A$ is a superring, then the identity element in $A$ is always even. If $A$ is supercommutative, then $a^{2}=0$ if $a$ is odd.  We write $\textit{sRings}$ for the category of superrings. Every ordinary ring is a superring whose odd part is trivial. This gives us a fully faithful functor $\textit{Rings}\to\textit{sRings}$.

An \textit{ideal} $I$ of a superring $A$ is a $\mathbb{Z}_2$-graded ideal of $A$ such that $ax\in I$ and $xa\in I$ for every $a\in A$ and $x\in I$. An ideal $I$ is \textit{finitely generated} if it is generated by finitely many homogeneous elements. Note that every superring contains a canonical ideal $J_{A}$ generated by all the odd elements. The \textit{odd dimension} of a superring $A$ is the smallest number of generators of the ideal $J_{A}$. The quotient ring $A/J_{A}$ is called the \textit{bosonic reduction} of $A$, which is an ordinary ring. This gives us a left adjoint to the embedding from $\textit{Rings}$ to \textit{sRings}. We say a proper ideal $\mathfrak{p}\subseteq A$ is \textit{prime} if $ab\in\mathfrak{p}$ then either $a\in\mathfrak{p}$ or $b\in\mathfrak{p}$ for every homogeneous element $a,b\in A$. An ideal $\mathfrak{m}\subseteq A$ is maximal if it is maximal among proper ideals of $A$. A \textit{local superring} is a superring with only one maximal ideal. Note that every prime ideal of $A$ contains $J_{A}$ because every odd element is nilpotent.

Throughout this article, we will assume that all superrings are supercommutative. In particular, we do not require superrings to have finite odd dimensions. However, every finite set of odd elements is contained in a nilpotent ideal generated by finitely many odd elements.

\begin{definition}\label{def:supermodules}
    Let $A$ be a superring. A left (resp. right) $A$-\textit{module} is a $\mathbb{Z}_{2}$-graded abelian group $M=M_{0}\oplus M_{1}$ with a left (resp. right) $A$-action, that is, a morphism $l\colon A\times M\to M$ (resp. $l\colon M\times A\to M$) such that $A_{i}\times M_{j}\subseteq M_{i+j}$ for all $i,j\in\{0,1\}$. A left (resp. right) $A$-module homomorphism is a morphism of $\mathbb{Z}_{2}$-graded abelian groups that preserves the left (resp. right) $A$-actions.
\end{definition}

Since $A$ is supercommutative, every left $A$-module admits a canonical right $A$-module structure. One checks that these two actions are compatible with each other. By an $A$-module, we mean a left $A$-module together with the canonical right $A$-action. Therefore, we will not distinguish left and right $A$-modules and simply refer to them as $A$-modules. Any ideal $I$ of $A$ is a submodule of $A$. An $A$-module is finitely generated if it can be generated as an ordinary $A$-module by finitely many homogeneous elements. The usual notion of tensor products readily extends to the super case. An \textit{$A$-superalgebra} is a superring $B$ equipped with a morphism of superrings $A\to B$. It is finitely generated if it is generated as an $A$-superalgebra by finitely many homogeneous elements. A superring is \textit{Noetherian} if it satisfies the ascending chain condition on its ideals. The usual properties of Noetherian rings extend to the super case. For instance, every finitely generated algebra over a Noetherian superring is Noetherian. Let $f\colon A\to B$ be a morphism of superrings. We say $f$ is of \textit{finite type} if $B$ is finitely generated as an $A$-superalgebra. We say $f$ is \textit{finite} if $B$ is finitely generated as an $A$-module. This is consistent with the usual definition of finite type and finite morphisms of rings. One of the only properties of superrings that is not discussed in Westra's thesis but is essential for us is integrality. In \cite{RTT23}, a notion of integral extensions of superrings was introduced. However, it is not well-behaved in our setting. We will introduce a slightly different notion of integral extensions in \Cref{sec:integral}.

\subsection{Superschemes}
We will now recall some basic definitions in algebraic supergeometry. For more details on the foundation, we refer to \cite{BHP23}. Recall that a locally ringed space is a topological space with a sheaf of rings whose stalk at each point is a local ring. We define locally superringed spaces in a similar fashion.

\begin{definition}\label{def:superspaces}
    Let $X$ be a topological space and $\mathcal{O}_{X}=\mathcal{O}_{X,0}\oplus\mathcal{O}_{X,1}$ be a sheaf of superrings on $X$. We say $(X, \mathcal{O}_{X})$ is a \textit{locally superringed space} if for every point $x\in X$, the stalk $\mathcal{O}_{X,x}$ is a local superring. 
\end{definition}

A morphism of locally superringed spaces is just a pair $(f,f^{\#})\colon(X,\mathcal{O}_{X})\to(Y,\mathcal{O}_{Y})$ where $f\colon X\to Y$ is a continuous map and $f^{\#}\colon\mathcal{O}_{Y}\to f_{*}\mathcal{O}_{X}$ is a morphism of sheaves of superrings such that for every point $x\in X$, the induced map $\mathcal{O}_{Y,f(x)}\to\mathcal{O}_{X,x}$ is a local homomorphism of local superrings. It is clear from the definition that the category of locally ringed spaces embeds into the category of locally superringed spaces.

Let $A$ be a superring. We define the superspectrum $\operatorname{sSpec}(A)$ of $A$ to be the set of prime ideals of $A$. Similar to the ordinary case, we could equip $\operatorname{sSpec}(A)$ with the Zariski topology. This gives us a locally superringed space $\operatorname{sSpec}(A)$. See \cite[Section\ 5.4.1]{Westra:2009} for the details. We say a locally superringed space $X$ is an \textit{affine superscheme} if it is isomorphic to $\operatorname{sSpec}(A)$ for some superring $A$ in the category of locally superringed spaces.

\begin{definition}\label{def:superschemes}
    A \textit{superscheme} is a locally superringed space that is locally isomorphic to an affine superscheme.
\end{definition}

Similar to the ordinary case, a superscheme is \textit{locally Noetherian} if it is locally isomorphic to the spectrum of a Noetherian superring. A superscheme is \textit{Noetherian} if it is quasicompact and locally Noetherian. It is clear from the definition that the bosonic reduction of a superscheme is an ordinary scheme. Most properties of schemes and morphisms of schemes carry over naturally to the super case. See \cite[Appendix\ A]{BHP23} for a list of properties of superschemes carried over from the ordinary case. One of the few exceptions is projectivity, which becomes rather complicated in the super case. See \cite[Section\ 2]{BHP23} for the details. Alternatively, we could define a superscheme as a Zariski sheaf on the category of affine superschemes that admits a cover by representable affine open subsheaves. One can readily check that this is equivalent to \Cref{def:superschemes} from the functor-of-points perspective.

Let $(X,\mathcal{O}_{X})$ be a superscheme. We see that there is a sheaf of ideals $\mathcal{J}=(\mathcal{O}_{X,1})^2\oplus\mathcal{O}_{X,1}$ generated by the odd elements. Observe that the quotient $\mathcal{O}_{X}/\mathcal{J}$ is a sheaf of ordinary rings and $(X,\mathcal{O}_{X}/\mathcal{J})$ is an ordinary scheme. We say $(X,\mathcal{O}_{X}/\mathcal{J})$ is the \textit{bosonic reduction} of $(X,\mathcal{O}_{X})$ denoted by $X_{\text{bos}}$. Note that the quotient map $\mathcal{O}_{X}\to\mathcal{O}_{X}/\mathcal{J}$ induces a closed immersion $X_{\text{bos}}\to X$. This defines a functor from the category of superschemes to the category of schemes. One readily checks that this functor is the right adjoint of the natural embedding from the category of schemes to that of superschemes.

On the other hand, there is a $\Gamma=\{\pm 1\}$-action of $\mathcal{O}_{X}$ given by the parity involution, that is, $(-1)\cdot f=f_{0}-f_{1}$ for every element $f\in\mathcal{O}_{X}$. By definition, the superring of invariants $\mathcal{O}_{X,0}=(\mathcal{O}_{X})^\Gamma$ is another sheaf of ordinary rings. This produces another ordinary scheme $(X, \mathcal{O}_{X,0})$ from $(X, \mathcal{O}_{X})$ called the $\textit{bosonic quotient}$ of $X$ denoted by $X_{\text{ev}}$. One checks that taking the bosonic quotient is the left adjoint to the natural embedding. In general, bosonic reductions are very well-behaved. For example, they commute with fibre products. This fails for bosonic quotients. Let $X=\operatorname{sSpec}(\mathbb{C}[\theta])$ where $\theta$ is an odd variable. Then $X_{ev}$ is just $\operatorname{Spec}(\mathbb{C})$. However, the bosonic quotient of $X\times_{\mathbb{C}}X=\operatorname{sSpec}(\mathbb{C}[\theta_1,\theta_2])$ is isomorphic to the dual numbers $\operatorname{Spec}(\mathbb{C}[x]/(x^2))$. We also see that the bosonic quotient of a smooth superscheme need not be smooth.

\begin{definition}\label{def:superdim}
    Let $X$ be a superscheme. The even dimension of $X$ is the dimension of its bosonic reduction $X_{\text{bos}}$. The odd dimension of $X$ is the supremum of the odd dimensions of the local superrings $\mathcal{O}_{X,x}$ at every point $x$ of $X$. We say $X$ has dimension $(m|n)$ if it has even dimension $m$ and odd dimension $n$.
\end{definition}
Note that both the odd and even dimensions can be infinite for a superscheme.

\subsection{Superstacks and supergroupoids}\label{sec:stacks_groupoids}

In this subsection, we recall some terminology on superstacks and supergroupoids. For the foundations of superstacks, we generally follow the conventions and terminology in \cite{BH25}. For the definition of categories fibred in groupoids and superstacks, we refer to \cite[Section\ 3]{CV19}. Let $S$ be a superscheme. Let $\text{sSch}_{/S}$ be the category of superschemes over $S$. 

\begin{definition}\label{def:alg_superspaces}
    An \textit{algebraic superspace} $X$ over $S$ is a sheaf on $\text{sSch}_{/S}$ for the \'etale topology such that
    \begin{enumerate}[topsep=0pt,noitemsep,label=\normalfont(\arabic*)]
        \item the diagonal morphism $\Delta_{X}\colon X\to X\times_{S} X$ is representable by a superscheme over $S$, and
        \item there exists a \'etale surjective morphism $U\to X$, where $U$ is a superscheme over $S$.
    \end{enumerate}
\end{definition}

Observe that the diagonal map $\Delta_{X}$ is a monomorphism. Let $U\to X$ be an \'etale covering. Then we have the following Cartesian diagram.

\begin{equation}
    \begin{tikzcd}
        R\arrow[r,hookrightarrow]\arrow[d] & U\times_{S}U\arrow[d]\\
        X\arrow[r,hookrightarrow,"\Delta_{X}"] & X\times_{S} X.
    \end{tikzcd}
\end{equation}
Note that the vertical maps are \'etale by definition. This tells us that $R$ defines an \'etale equivalence relation on $U$ and $X$ is isomorphic to the quotient sheaf $U/R$ as \'etale sheaves.

\begin{definition}\label{def:alg_stacks}
    Let $S$ be an algebraic superspace. An \textit{algebraic superstack} $\mathcal{X}$ over $S$ is a stack over $\text{sSch}_{/S}$ for the \'etale topology, in the sense of \cite[Definition\ 3.3]{CV19}, such that
    \begin{enumerate}[topsep=0pt,noitemsep,label=\normalfont(\arabic*)]
        \item the diagonal morphism $\Delta_{\mathcal{X}}\colon\mathcal{X}\to \mathcal{X}\times_{S}\mathcal{X}$ is representable by an algebraic superspace over $S$, and
        \item there exists a smooth surjective representable morphism $U\to\mathcal{X}$, where $U$ is a superscheme over $S$.
    \end{enumerate}
    An algebraic superstack over $S$ is \textit{Deligne--Mumford} if it admits an \'etale surjective morphism from a superscheme over $S$.
\end{definition}
Note that the diagonal map $\Delta_{\mathcal{X}}$ is no longer a monomorphism when $\mathcal{X}$ is an algebraic superstack. However, we still have the following Cartesian diagram
\begin{equation}
    \begin{tikzcd}
        R\arrow[r]\arrow[d] & U\times_{S}U\arrow[d]\\
        \mathcal{X}\arrow[r,"\Delta_{\mathcal{X}}"] & \mathcal{X}\times_{S}\mathcal{X}.
    \end{tikzcd}
\end{equation}
In this case, one checks that $R\rightrightarrows U$ is a pre-equivalence relation in algebraic superspaces over $S$. Note that the vertical map $R\to\mathcal{X}$ is smooth and hence locally of finite presentation. For every algebraic superstack $\mathcal{X}$, the diagonal map $\Delta_{\mathcal{X}}$ is locally of finite type \cite[Proposition 3.68(1)]{BH25}. Every algebraic superstack $\mathcal{X}$ admits a bosonic reduction $\mathcal{X}_{\operatorname{bos}}$, which is an ordinary algebraic stack. Bosonic quotients exist for complex Deligne--Mumford superstacks \cite[Theorem\ A]{CV19}. For algebraic superstacks, however, the existence of bosonic quotients is a much subtler problem. See \cite{CV19} for more details. We say $\mathcal{X}$ is \textit{projected} if the canonical closed immersion $\mathcal{X}_{\operatorname{bos}}\to\mathcal{X}$ admits a retraction. Similar to the ordinary case, we have the following characterization of Deligne-Mumford superstacks and algebraic superspaces in terms of their diagonal.

\begin{proposition}[{cf.\ \cite[Proposition\ 3.68\ and\ 3.70]{BH25}}]\label{prop:DM_diagonal}
    An algebraic superstack is Deligne--Mumford if and only if its diagonal is unramified. An algebraic superstack is an algebraic superspace if and only if its diagonal is a monomorphism.
\end{proposition}

Let $x$ be a point of $\mathcal{X}$ and let $x\colon\operatorname{Spec}k\to\mathcal{X}$ be a representative. Consider the following Cartesian diagrams
\begin{equation}
    \begin{tikzcd}
        G_{x}\arrow[r]\arrow[d] & \operatorname{Spec}k\arrow[d]\\
        \mathcal{X}\arrow[r,"\Delta"] & \mathcal{X}\times_{S}\mathcal{X},
    \end{tikzcd}
    \begin{tikzcd}
        \mathcal{I}_{\mathcal{X}}\arrow[d]\arrow[r] & \mathcal{X}\arrow[d,"\Delta"]\\
        \mathcal{X}\arrow[r,"\Delta"] & \mathcal{X}\times_{S}\mathcal{X}.
    \end{tikzcd}
\end{equation}
We say $G_{x}$ is the stabilizer of $x$, which is a group algebraic superspace by definition. We call $\mathcal{I}_{\mathcal{X}}$ the inertia superstack of $\mathcal{X}$. If $P$ is a property of morphisms of superschemes that is stable under base change and local on the target, then we say $\mathcal{X}$ \textit{has $P$ inertia} if the morphism $\mathcal{I}_{\mathcal{X}}\to\mathcal{X}$, which is representable, is $P$. For example, we say $\mathcal{X}$ has finite inertia if $\mathcal{I}_{\mathcal{X}}\to\mathcal{X}$ is finite.

We now recall the notion of supergroupoids. This was first introduced in \cite{CV19} over the complex numbers.

\begin{definition}\label{def:supergroupoids}
    Let $S$ be an algebraic superspace. A supergroupoid over $S$ is a quintuple $(U,R,s,t,c)$ where $U$ and $R$ are algebraic superspaces over $S$ and $s,t\colon R\to U$ and $c\colon R\times_{s,U,t}R\to R$ are morphisms of algebraic superspaces over $S$ such that for every superscheme $T$ over $S$, the quintuple $(U(T),R(T),s,t,c)$ is a groupoid in the sense of \cite[\href{https://stacks.math.columbia.edu/tag/043V}{Tag 043V}]{stacks-project}. A morphism of supergroupoids over $S$ is a pair $(u,r)\colon(U,R,s,t,c)\to(U^\prime,R^\prime,s^\prime,t^\prime,c^\prime)$ where $u\colon U\to U^\prime$ and $r\colon R\to R^\prime$ are morphisms of algebraic superspaces over $S$ that induce a functor from $(U(T),R(T),s,t,c)$ to $(U^\prime(T),R^\prime(T),s^\prime,t^\prime,c^\prime)$ for every superscheme $T$ over $S$. Let $P$ be a property of morphisms of algebraic superspaces (e.g.\ flat, smooth, \'etale, separated, finite, etc). We say a supergroupoid $(U,R,s,t,c)$ is $P$ if $s$ and $t$ are $P$ over $S$.
\end{definition}

We refer to the morphism $c$ in \Cref{def:supergroupoids} as the composition map of the supergroupoid. For simplicity, we will write $s,t\colon R\rightrightarrows U$ for a supergroupoid. Let $j=(s,t)\colon R\to U\times_{S}U$. Let $u\colon T\to U$ be a $T$-valued point of $U$ where $T$ is a superscheme over $S$. Then the \textit{stabilizer} of $u$, denoted by $\operatorname{stab}(u)$, is defined to be the pullback of the induced map $T\to U\times_{S}U$ along $j$. Let $f=(f_R,f_U)\colon (R_U, U)\to(R_W,W)$ be a morphism of supergroupoids. For a geometric point $u\colon\operatorname{Spec}k\to U$, we say $f$ is \textit{stabilizer-preserving at $u$} if the natural morphism $\operatorname{Stab}(u)\to\operatorname{Stab}(f_U\circ u)$ is an isomorphism. We say $f$ is \textit{stabilizer-preserving} if it is stabilizer-preserving at every geometric point $u$ of $U$.

Let $u$ be a point in the underlying topological space of $|U|$. The \textit{orbit} of $u$, denoted by $O(u)$, is defined as the subset $t(s^{-1}(u))$ of $|U|$. A subset $Z\subseteq|U|$ is \textit{$R$-invariant} if for every $z\in Z$ and for every $r\in R$ such that $t(r)=z$, we have $s(r)\in Z$. For example, the orbit of every point of $U$ is invariant by definition. Given any subset $Z\subseteq|U|$, one readily checks that the set $s(t^{-1}(Z))$ is $R$-invariant. Note that $Z$ is contained in $s(t^{-1}(Z))$. If $Z$ is closed with complement $V$, then the set $V^\prime=U\setminus s(t^{-1}(Z))$ is the largest saturated subset of $V$. We say $V$ is \textit{saturated} if $V^\prime=V$.

For every supergroupoid $s,t\colon R\rightrightarrows U$, there is an associated superstack $[U/R]$ in the sense of \cite[Section\ 2.2]{BH25}. See also \cite[\href{https://stacks.math.columbia.edu/tag/044Q}{Tag 044Q}]{stacks-project}. Conversely, if an algebraic superstack $\mathcal{X}$ admits a flat surjective morphism from a superscheme or algebraic superspace $U\to\mathcal{X}$ that is locally of finite presentation, then we have the following Cartesian diagram
\begin{equation}
    \begin{tikzcd}
        R\arrow[r,"{(s,t)}"]\arrow[d] & U\times U\arrow[d]\\
        \mathcal{X}\arrow[r,"\Delta_{\mathcal{X}}"] & \mathcal{X}\times\mathcal{X}.
    \end{tikzcd}
\end{equation}
This induces a flat supergroupoid $s,t\colon R\rightrightarrows U$ with the composition map $c\colon R\times_{s,U,t}R\to R$. In this case, the associated quotient stack $[U/R]$ is equivalent to $\mathcal{X}$. One readily checks this following the proof of \cite[\href{https://stacks.math.columbia.edu/tag/04T5}{Tag 04T5}]{stacks-project}, which is completely categorical. In this case, we say $s,t\colon R\rightrightarrows U$ is a groupoid presentation of $\mathcal{X}$.

We now introduce quotients by supergroupoids following the ordinary case in \cite[Section\ 2]{Ryd13}. Let $s,t\colon R\rightrightarrows U$ be a supergroupoid. We say a morphism $q\colon U\to X$ of algebraic superspaces is \textit{equivariant} if $q\circ s=q\circ t$. Let $f\colon X^\prime\to X$ be another morphism. Base changing along $f$ gives another supergroupoid $s^\prime,t^\prime\colon R^\prime\rightrightarrows U^\prime$ with an equivariant map $q^\prime\colon U^\prime\to X^\prime$. We say a property $P$ of $q$ is universal (resp. uniform) if it is stable under arbitrary (resp. flat) base change.

\begin{definition}[{cf.\ \cite[Definition 2.2]{Ryd13}}]\label{def:quotients}
    Let $s,t\colon R\rightrightarrows U$ be a supergroupoid and let $q\colon U\to X$ be an equivariant morphism. We say that $q$ is a
    \begin{enumerate}
        \item \textit{categorical quotient} if it is initial among morphisms from $U$ to an algebraic superspace. That is, if $q^\prime\colon U\to Y$ is another $R$-equivariant morphism with $q^\prime\circ s=q^\prime\circ t$, then there exists a unique morphism $\phi\colon X\to Y$ such that $q^\prime=\phi\circ q$.
        \item \textit{topological quotient} if the map $U\to X$ is the coequalizer of $|R|\rightrightarrows|U|$ in the category of topological spaces with respect to both the Zariski and constructible topology.
        \item \textit{universal topological quotient} if for every morphism $X^\prime\to X$, the base change map $U^\prime\to X^\prime$ is the coequalizer of $|R^\prime|\rightrightarrows|U^\prime|$ in the category of topological spaces with respect to both the Zariski and constructible topology.
        \item \textit{strongly topological quotient} if it is a topological quotient and the map $j_{/X}=(s,t)\colon R\to U\times_X U$ is universally submersive.
        \item \textit{geometric quotient} if it is universally topological and $\mathcal{O}_X\cong(q_{*}\mathcal{O}_U)^{R}$.
        \item \textit{strong geometric quotient} if it is strongly topological and geometric.
    \end{enumerate}
\end{definition}

\begin{proposition}[{cf.\ \cite[Proposition 2.10]{Ryd13}}]\label{prop:geometric_quotient_uniform_local}
    Let $s,t\colon R\rightrightarrows U$ be a supergroupoid and let $q\colon U\to X$ be an equivariant morphism. Let $f\colon X^\prime\to X$ be a morphism. Let $q^\prime\colon U^\prime\to X^\prime$ be the base change of $q$.
    \begin{enumerate}[topsep=0pt,noitemsep,label=\normalfont(\arabic*)]
        \item\label{prop:geometric_quotient_uniform_local1} If $f$ is flat and $q$ is a (resp. strongly) geometric quotient, then so is $q^\prime$.
        \item\label{prop:geometric_quotient_uniform_local2} If $f$ is a covering in the fpqc topology and $q^\prime$ is a topological (resp. strongly topological, geometric, strongly geometric) quotient, so is $q$.
    \end{enumerate}
\end{proposition}

\begin{proof}
    Consider the following sequence of sheaves for the \'etale topology
    \begin{equation*}
        \begin{tikzcd}
            \mathcal{O}_X\arrow[r] & q_*\mathcal{O}_U\arrow[r,shift left,"s^{*}"]\arrow[r,shift right,"t^{*}",swap] & (q\circ s)_*\mathcal{O}_R.
        \end{tikzcd}
    \end{equation*}
    \ref{prop:geometric_quotient_uniform_local1} holds because the exactness of this sequence is compatible with flat base change. For \ref{prop:geometric_quotient_uniform_local2}, we first note that coverings in the fpqc topology are submersive in both the Zariski and constructible topology. Therefore the statement holds for topological and strongly topological quotients. The statement for geometric and strongly geometric quotients then follows from the fact that the exactness of the equalizer sequence above is local in the fpqc topology.
\end{proof}

We wish to remark that a (strongly) geometric quotient is always uniform but not necessarily universal. Moreover, one readily checks that being a (strong) geometric quotient is a local property in fpqc topology. A priori, topological quotients and geometric quotients are not categorical, however. We will show in \Cref{sec:descent_condition} that a strong geometric quotient that is integral satisfies the descent condition, and thus is categorical.

\section{\'Etale descent in the super case}\label{sec:descent}

\subsection{Integral extensions of superrings}\label{sec:integral}

As remarked earlier, the theory of commutative superalgebras is developed in detail in Westra's thesis \cite{Westra:2009}. One of the few concepts not discussed is the integral extension of superrings. We propose a definition of integral morphisms, which is sufficient for our application. We will also extend properties of ordinary integral morphisms to the super case. We begin with the definition of an integral morphism of superrings.

\begin{definition}\label{def:integral_maps}
    Let $\phi\colon A\to B$ be a morphism of superrings. An element $b\in B$ is integral over $A$ if there exists a monic even polynomial $P(x)\in A_{0}[x]$ such that $\phi(P)(b)=0$, where $\phi(P)(x)$ is the image of $P(x)$ under the map $A[x]\to B[x]$ induced by $\phi$. We say $\phi$ is integral if every element $b\in B$ is integral over $A$. If $\phi$ is integral and injective, then we say $\phi$ is an integral extension of superrings.
\end{definition}

\begin{lemma}\label{lem:integral_bosonic}
    Let $\phi\colon A\to B$ be a morphism of superrings. An element $b=b_0+b_1\in B$ is integral over $A$ if and only if $b_0$ is integral over $A_0$. Moreover, $\phi\colon A\to B$ is integral if and only if the induced map $\phi_0\colon A_0\to B_0$ is integral.
\end{lemma}

\begin{proof}
    Let $b=b_0+b_1$ be an element in $B$. For every even polynomial $P(x)\in B_0[x]$, we have$$P(b)=P(b_0)+b_1P^\prime(b_0).$$If $b$ is integral over $A$, then there exists a monic even polynomial $P(x)\in A_{0}[x]$ such that $\phi(P)(b)=0$. Then we have$$P(b_0)+b_1P^\prime(b_0)=\phi(P)(b)=0$$as elements in $B$. It follows that $P(b_0)=0$ as desired. Conversely, suppose $b_0$ is integral over $A_0$. Then there exists a monic even polynomial $P(x)\in A_{0}[x]$ such that $\phi(P)(b_0)=0$. We claim that $\phi(P^2)(b)=0$. Indeed, this follows from the fact that $\phi(P)(b)=b_1\phi(P^\prime)(b_0)$, and we have$$\phi(P^2)(b)=(\phi(P)(b))^2=(b_1\phi(P^\prime)(b_0))^2=0$$since $b_1^2=0$. This finishes the proof.
\end{proof}

An immediate consequence of our definition is the following characterization of integral elements.

\begin{lemma}\label{lem:integral_module}
    Let $\phi\colon A\to B$ be a morphism of superrings. An element $b\in B$ is integral over $A$ if and only if there exists a finitely generated $A$-submodule $M$ of $B$ such that $1_B\in M$ and $bM\subseteq M$.
\end{lemma}

\begin{proof}
    Choose homogeneous generators $m_0=1,m_1,\cdots,m_n$ of $M$ as an $A$-module. Since $bM\subseteq M$, we must have $b_0M\subseteq M$. It follows that for every $0\leq i\leq n$, we have$$b_0m_i=\sum_{j}a_{ij}m_j$$for some homogeneous elements $a_{ij}\in A$. Note that there are only finitely many $a_{ij}$ that are odd for every $j$. Let $A^\prime\subseteq A$ be the $A_0$-superalgebra generated by the odd $a_{ij}$'s. Let $M^\prime$ be the $A^\prime$-module generated by $m_0=1,m_1,\cdots,m_n$. By construction, $M^\prime$ is a finite $A_0$-module and $b_0M^\prime\subseteq M^\prime$. We are now in the ordinary case and \cite[\href{https://stacks.math.columbia.edu/tag/05BT}{Tag 05BT}]{stacks-project} implies that $b_0$ is integral. By \Cref{lem:integral_bosonic}, $b$ is also integral. Conversely, suppose $b\in B$ is integral over $A$. Then there exists a monic even polynomial $P(x)\in A_{0}[x]$ such that $\phi(P)(b)=0$. Let $M$ be the $A$-module generated by $\{b_0^i,b_1b_0^i\}_{0\leq i<d}$ where $d$ is the degree of $P(x)$. One readily checks that $1_B\in M$ and $bM\subseteq M$, and the result follows.
\end{proof}

Integral extensions of superrings and, more generally, noncommutative rings are also discussed in \cite{ATT72} and \cite{RTT23}. While \cite[Definition\ 5.1]{RTT23} is more general, we believe \Cref{def:integral_maps} behaves better in the geometric setting we are working in. The following lemma relates integral and finite morphisms of superrings.

\begin{lemma}\label{lem:integral+finite_type=finite}
    Let $\phi\colon A\to B$ be a morphism of superrings. Then the following are equivalent.
    \begin{enumerate}[topsep=0pt,noitemsep,label=\normalfont(\arabic*)]
        \item $\phi\colon A\to B$ is finite,
        \item $\phi$ is integral and of finite type, and
        \item $B$ is generated as an $A$-superalgebra by finitely many elements that are integral over $A$.
    \end{enumerate}
\end{lemma}

\begin{proof}
     (2) $\implies$ (3) is immediate. (1) $\implies$ (2) follows from applying \Cref{lem:integral_module} to $M=B$. For (3) $\implies$ (1), let $b_{1},\cdots,b_{n}$ be a set of homogeneous generators of $B$ as an $A$-superalgebra that are integral over $A$. Since $b_{i}$ is integral over $A$, it satisfies a monic polynomial $P_{i}(x)\in A_{0}[x]$. Let $d_{i}$ be the degree of $P_{i}(x)$. Then the set$$\{b_{1}^{s_{1}}\cdots b_{n}^{s_{n}}\ |\ 0\leq s_{i}< d_{i},\ 1\leq i\leq n\}$$generates $B$ as an $A$-module.
\end{proof}

We also record the following variant of the Artin--Tate lemma for superrings, which is essential in the upcoming sections of this article.

\begin{lemma}\label{lem:super_artin_tate}
    Let $A\subseteq B\subseteq C$ be extensions of superrings. If $A$ is Noetherian and $C$ is a finitely generated $A$-superalgebra that is integral over $B$, then $B$ is also a finitely generated $A$-superalgebra.
\end{lemma}

\begin{proof}
    Since $A$ is a Noetherian superring, $A_0$ is a Noetherian ring. Since $C$ is a finitely generated $A$-superalgebra, $C_0$ is a finitely generated $A_0$-algebra. It follows that $C_0$ is a finitely generated $B_0$-algebra. Indeed, we have $A_0\subseteq B_0$, and the same set of $A_0$-algebra generators will automatically be a set of $B_0$-algebra generators. By \Cref{lem:integral_bosonic}, we see that $C_0$ is integral over $B_0$. It follows that $C_0$ is a finitely generated $B_0$-module. The ordinary Artin--Tate lemma tells us that $B_0$ is a finitely generated $A_0$-algebra. It follows that $B_0$ is Noetherian. Note that $C$ is also Noetherian as it is finitely generated over $A$. It follows that $C_1$ is a finitely generated $C_0$-module and thus a finitely generated $B_0$-module. Therefore $B_1$ must be a finitely generated $B_0$-module because it is a $B_0$-submodule of a finitely generated $B_0$-module. It follows that $B$ is a finitely generated $A$-algebra as desired. 
\end{proof}

The following lemma tells us that integral extensions of superrings induce closed and surjective morphisms on the spectra. 

\begin{lemma}[{cf.\ \cite[\href{https://stacks.math.columbia.edu/tag/00GQ}{Tag 00GQ}]{stacks-project}}]\label{lem:integral_surjective}
    Let $A\to B$ be an integral extension of superrings. Then the induced map $\phi^{*}\colon\operatorname{sSpec}(B)\to\operatorname{sSpec}(A)$ is universally closed and surjective. 
\end{lemma}

\begin{proof}
    First of all, integral homomorphism of superrings are stable under base change. This follow verbatim from the proof of \cite[\href{https://stacks.math.columbia.edu/tag/02JK}{Tag 02JK}]{stacks-project} using \Cref{lem:integral+finite_type=finite}. For any integral homomorphism $A\to B$, we have $A_0\to B_0$ is also integral by \Cref{lem:integral_bosonic}. It follows that the induced map $\phi_0^*\colon\operatorname{Spec}(B_0)\to\operatorname{Spec}(A_0)$ is universally closed by \cite[\href{https://stacks.math.columbia.edu/tag/01WM}{Tag 01WM}]{stacks-project}. Since $A\to B$ is an integral extension, so is $A_0\to B_0$. It follows from \cite[\href{https://stacks.math.columbia.edu/tag/00GQ}{Tag 00GQ}]{stacks-project} that $\phi_0^*$ is surjective. Moreover, we have $|\operatorname{sSpec}(A)|=|\operatorname{Spec}(A_0)|$ and $|\operatorname{sSpec}(B)|=|\operatorname{Spec}(B_0)|$. Therefore $\phi^*$ is closed and surjective as desired.
\end{proof}

As in the ordinary case, we define integral morphisms of superschemes and algebraic superspaces as follows.

\begin{definition}\label{def:integral_morphism_super}
    Let $f\colon X\to Y$ be a morphism of superschemes. We say $f$ is integral if it is affine and for every affine open subset $\operatorname{sSpec}(A)=U\subseteq Y$ with preimage $\operatorname{sSpec}(B)=f^{-1}(U)\subseteq X$, the induced map of superrings $A\to B$ is integral. 
\end{definition}

\begin{lemma}\label{lem:integral_is_affine+universally_closed}
    A morphism of superschemes is integral if and only if it is affine and universally closed. An integral morphism of superschemes of finite type is finite.
\end{lemma}

\begin{proof}
    Let $f\colon X\to Y$ be an affine and universally closed morphism of superschemes. We show that $f$ is integral. This is Zariski-local on $Y$, and we may assume that $X$ and $Y$ are affine superschemes. Consider the bosonic reduction $f_{\operatorname{bos}}\colon X_{\operatorname{bos}}\to Y_{\operatorname{bos}}$. Note that $f_{\operatorname{bos}}$ is affine and universally closed, and thus integral. The statement translates into the following algebra problem: if $A\to B$ is a homomorphism of superrings such that $A_0/A_1^2\to B_0/B_1^2$ is integral, then $A\to B$ is integral. To this end, let $b_0\in B_0$, and we claim that there exists a monic polynomial $P\in A_0[x]$ such that $P(b_0)\in B_1^2$. Indeed, let $\overline{b}_0$ be in the image of $b_0$ in $B_0/B_1^2$. Then there exists a monic polynomial $\overline{P}\in (A_0/A_1^2)[x]$ such that $\overline{P}(\overline{b}_0)=0$. We may then lift $\overline{P}$ to a desired polynomial $P$. Since $P(b_0)\in B_1^2$, we have $P(b_0)^N=0$ for some $N>0$. It follows that $A_0\to B_0$ is integral and the result follows from \Cref{lem:integral_bosonic}. The rest of the lemma follows immediately from \Cref{def:integral_morphism_super}, \Cref{lem:integral_surjective}, and \Cref{lem:integral+finite_type=finite}.
\end{proof}

Note that affineness and universal closedness are both stable under base change, fpqc-local on the target, and satisfy fpqc descent.  \Cref{lem:integral_is_affine+universally_closed} implies that integral morphisms also enjoy these properties. We may then extend the definition of integral morphisms from superschemes to algebraic superspaces.

\begin{definition}\label{def:integral_morphism_superspaces}
    Let $f\colon X\to Y$ be a morphism of algebraic superspaces. We say $f$ is integral if for every morphism $U\to Y$ from an affine superscheme, the map $X\times_YU\to U$ is representable by an affine superscheme integral over $U$. 
\end{definition}

\begin{lemma}\label{lem:integral_morphism_etale_local}
    Let $S$ be a superscheme. Let $f\colon X\to Y$ be a morphism of algebraic superspaces over $S$. Then the following statements are equivalent.
    \begin{enumerate}
        \item\label{lem:integral_morphism_etale_local1} $f$ is representable and integral,
        \item\label{lem:integral_morphism_etale_local2} $f$ is integral
        \item\label{lem:integral_morphism_etale_local3} There exists a surjective \'etale morphism $V\to Y$ from a superscheme such that $X\times_YV\to V$ is integral.
    \end{enumerate}
\end{lemma}

\begin{proof}
    (\ref{lem:integral_morphism_etale_local1})$\implies$(\ref{lem:integral_morphism_etale_local2}) and (\ref{lem:integral_morphism_etale_local2})$\implies$(\ref{lem:integral_morphism_etale_local3}) are clear. For (\ref{lem:integral_morphism_etale_local3})$\implies$(\ref{lem:integral_morphism_etale_local1}), let $T\to Y$ be a morphism from an affine superscheme $T$. Let $T_V=T\times_YV$. Then $T_V\to T$ is a \'etale cover and $(X\times_YT)\times_TT_V\to T_V$ is affine and integral. By \'etale descent of affine morphisms, we see that $X\times_YT\to T$ is affine. By \Cref{lem:integral_is_affine+universally_closed}, we just need to show $X\times_YT\to T$ is universally closed. But this is \'etale local on the target. Since this is true for every such morphism $T\to Y$, the result follows.
\end{proof}

In particular, the bosonic reduction of an integral morphism of algebraic superspaces is also integral. Indeed, the statement is \'etale local on the target by \Cref{lem:integral_morphism_etale_local}. So we immediately reduce to the affine case, and the statement follows from \Cref{lem:integral_bosonic}.

\subsection{Effective descent for \'etale morphisms of superspaces}\label{sec:effective_descent}

In this subsection, we establish some effective descent results for \'etale morphisms in the category of algebraic superspaces. Throughout this section, we will generally follow the terminology used in \cite[Section\ 5]{Ryd10}.

Let $S$ be an algebraic superspace and let $P$ be a property of morphisms of algebraic superspaces. We write $\mathbf{\acute{E}t}_{P}(S)$ for the category consisting of \'etale morphisms $X\to S$ that satisfy property $P$. For example, we write $\mathbf{\acute{E}t}_{\operatorname{sep}}(S)$ for the category of separated and \'etale morphisms over $S$. Let $f\colon S^\prime\to S$ be another morphism of algebraic superspaces. Let $S^{\prime\prime}=S^\prime\times_SS^\prime$. Let $p_1,p_2\colon S^{\prime\prime}\to S^\prime$ be the canonical projections. We write $\mathbf{\acute{E}t}_{P}(f\colon S^\prime\to S)$ for the category consisting of pairs $(X^\prime\to S^\prime, \varphi)$ where $X^\prime\to S^\prime\in\mathbf{\acute{E}t}_{P}(S^\prime)$ and $\varphi\colon X^\prime\times_{S^\prime,p_1}S^{\prime\prime}\cong X^\prime\times_{S^\prime,p_2}S^{\prime\prime}$ is a descent datum satisfying the cocycle condition (e.g. \cite[page 19]{Ryd10}). By descent theory, there is a natural functor $f_P^*\colon\mathbf{\acute{E}t}_{P}(S)\to\mathbf{\acute{E}t}_{P}(f\colon S^\prime\to S)$ taking a morphism $(X\to S)$ to $(X\times_{S}S^\prime\to S^\prime, \varphi)$ where $\varphi$ is the pullback of the trivial descent datum on $X$ along $f$. We say a descent datum $\varphi$ on $X^\prime\to S^\prime$ is $\mathbf{\acute{E}t}_{P}$-effective if the pair $(X^\prime, \varphi)$ lies in the essential image of $f_P^*$. We say a morphism $f\colon S^\prime\to S$ is of effective $\mathbf{\acute{E}t}_{P}$-descent if the functor $f_P^*$ is an equivalence. We say a morphism $f\colon S^\prime\to S$ is of universal effective $\mathbf{\acute{E}t}_{P}$-descent if this morphism is of effective $\mathbf{\acute{E}t}_{P}$-descent for every morphism $T\to S$. 

The goal of this section is to show that integral surjective morphisms of algebraic superspaces are of universal effective $\mathbf{\acute{E}t}_{\text{sep}}$-descent. We do so by reducing to the ordinary case where the statement is known by \cite[Corollary 5.15]{Ryd10}. Throughout this subsection, we will assume that $S$ is an algebraic superspace over $\mathbb{Z}[1/2]$.

\begin{proposition}[{cf.\ \cite[Proposition\ A.4]{Ryd10}}]\label{prop:super_une_equivalence}
    Let $S^\prime\to S$ be a surjective closed immersion of superspaces. Then the natural functor $\mathbf{\acute{E}t}(S)\to\mathbf{\acute{E}t}(S^\prime)$ sending $(X\to S)$ to $(X\times_{S}S^\prime\to S^\prime)$ is an equivalence of categories.
\end{proposition}

\begin{proof}
    For a surjective closed immersion of algebraic superspaces $S^\prime\to S$, we get a surjective closed immersion of ordinary algebraic spaces $S_{\operatorname{ev}}^\prime\to S_{\operatorname{ev}}$. By \Cref{lem:ev_commutes_etale_base_change}, we have $\mathbf{\acute{E}t}(S)\cong\mathbf{\acute{E}t}(S_{\operatorname{ev}})$ and $\mathbf{\acute{E}t}(S^\prime)\cong\mathbf{\acute{E}t}(S_{\operatorname{ev}}^\prime)$. It remains to show $\mathbf{\acute{E}t}(S_{\operatorname{ev}})\cong\mathbf{\acute{E}t}(S_{\operatorname{ev}}^\prime)$ and this is precisely \cite[Proposition\ A.4]{Ryd10}.
\end{proof}

Recall that the natural inclusion map $S_{\operatorname{bos}}\to S$ is a surjective closed immersion. \Cref{prop:super_une_equivalence} tells us that $\mathbf{\acute{E}t}(S_{\operatorname{bos}})\cong\mathbf{\acute{E}t}(S)$. Let $P$ be a property of morphisms of superschemes or algebraic superspaces. We say $P$ is stable under bosonic reductions if for every morphism $f\colon X\to S$ of algebraic superspaces, $f$ has $P$ if and only if the induced map $f_{\operatorname{bos}}\colon X_{\operatorname{bos}}\to S_{\operatorname{bos}}$ has $P$. For example, every topological and separation property of morphisms of superspaces, such as quasi-compactness, quasi-separatedness, and separateness, is stable under bosonic reductions. The following result is immediate from \Cref{prop:super_une_equivalence}.

\begin{corollary}\label{cor:super_same_ET}
    Let $S$ be an algebraic superspace. Let $P$ be a property of morphisms of algebraic superspaces that is stable under bosonic reduction. Then the natural functor $\mathbf{\acute{E}t}_P(S)\to\mathbf{\acute{E}t}_P(S_{\operatorname{bos}})$ sending $(X\to S)$ to $(X\times_{S}S_{\operatorname{bos}}\to S_{\operatorname{bos}})$ is an equivalence of categories.
\end{corollary}

We now consider the \'etale morphisms of superspaces with descent data. Recall that for an \'etale morphism $f\colon X\to S$, we have the following Cartesian diagram

\begin{equation*}
    \begin{tikzcd}
        X_{\operatorname{bos}}\arrow[r]\arrow[d]\arrow[phantom,"\square",rd] & X\arrow[d]\\
        S_{\operatorname{bos}}\arrow[r] & S.
    \end{tikzcd}
\end{equation*}

\begin{lemma}\label{lem:super_same_ET_descent_datum}
    Let $S$ be an algebraic superspace. Then the assignment
    \begin{align*}
        \mathbf{\acute{E}t}(f\colon S^\prime\to S)&\longrightarrow\mathbf{\acute{E}t}(f_{\operatorname{bos}}\colon S_{\operatorname{bos}}^\prime\to S_{\operatorname{bos}})\\
        (X^\prime\to S^\prime,\varphi)&\longmapsto(X^\prime\times_{S^\prime}S_{\operatorname{bos}}^\prime,\varphi_{\operatorname{bos}})
    \end{align*}
    induces an equivalence of categories. Let $f\colon S^\prime\to S$ be a morphism of algebraic superspaces. If $P$ is a property of morphisms of algebraic superspaces that is stable under bosonic reductions, then the equivalence above restricts to an equivalence between $\mathbf{\acute{E}t}_{P}(f\colon S^\prime\to S)$ and $\mathbf{\acute{E}t}_{P}(f_{\operatorname{bos}}\colon S_{\operatorname{bos}}^\prime\to S_{\operatorname{bos}})$. 
\end{lemma}

\begin{proof}
    Let $(X^\prime\to S^\prime, \varphi)\in\mathbf{\acute{E}t}(f\colon S^\prime\to S)$. We know that $X_{\operatorname{bos}}^\prime\to S_{\operatorname{bos}}^\prime$ is also \'etale and$$\varphi_{\operatorname{bos}}\colon X_{\operatorname{bos}}^\prime\times_{S_{\operatorname{bos}}^\prime}S_{\operatorname{bos}}^{\prime\prime}\cong X_{\operatorname{bos}}^\prime\times_{S_{\operatorname{bos}}^\prime}S_{\operatorname{bos}}^{\prime\prime}$$is an isomorphism. Since bosonic reductions commute with base change, we see that $\varphi_{\operatorname{bos}}$ satisfies the cocycle conditions. Conversely, let $(Y^\prime\to S_{\operatorname{bos}}^\prime, \phi)$ be an object in $\mathbf{\acute{E}t}(f_{\operatorname{bos}}\colon S_{\operatorname{bos}}^\prime\to S_{\operatorname{bos}})$. By \cref{cor:super_same_ET}, there is a unique \'etale morphism $X^\prime\to S^\prime$ such that $Y^\prime\cong X^\prime\times_{S^\prime}S_{\operatorname{bos}}^\prime$. Similarly, the isomorphism$$\phi^\prime\colon X_{\operatorname{bos}}^\prime\times_{S_{\operatorname{bos}}^\prime}S_{\operatorname{bos}}^{\prime\prime}\longrightarrow X_{\operatorname{bos}}^\prime\times_{S_{\operatorname{bos}}^\prime}S_{\operatorname{bos}}^{\prime\prime}$$over $S_{\operatorname{bos}}^{\prime\prime}\cong S_{\operatorname{bos}}^\prime\times_{S_{\operatorname{bos}}}S_{\operatorname{bos}}^\prime$ lifts to a unique isomorphism$$\phi\colon X^\prime\times_{S^\prime}S^{\prime\prime}\longrightarrow X^\prime\times_{S^\prime}S^{\prime\prime}$$over $S^{\prime\prime}$ satisfying the cocycle condition. One can readily check that these two functors are quasi-inverses of each other, and the result follows.
\end{proof}

Now we are ready to prove the main theorem of this section. 

\begin{theorem}[{cf.\ \cite[Theorem A.2(1a)]{Ryd13}}]\label{thm:effective_descent_integral_sep}
    Let $f\colon S^\prime\to S$ be an integral surjective morphism of algebraic superspaces. Then $f$ is of universal effective $\mathbf{\acute{E}t}_{\text{sep}}$--descent. 
\end{theorem}

\begin{proof}
    Since $f\colon S^\prime\to S$ is integral and surjective, it remains so after arbitrary base change. It suffices to show $f$ is of effective $\mathbf{\acute{E}t}_{\text{sep}}$--descent. We show that $\mathbf{\acute{E}t}_{\text{sep}}(S)\cong\mathbf{\acute{E}t}_{\text{sep}}(f\colon S^\prime\to S)$. By \Cref{prop:super_une_equivalence}, we have $\mathbf{\acute{E}t}_{\text{sep}}(S)\cong\mathbf{\acute{E}t}_{\text{sep}}(S_{\operatorname{bos}})$. By \Cref{lem:super_same_ET_descent_datum}, we have$$\mathbf{\acute{E}t}_{\text{sep}}(f\colon S^\prime\to S)\cong\mathbf{\acute{E}t}_{\text{sep}}(f_{\operatorname{bos}}\colon S_{\operatorname{bos}}^\prime\to S_{\operatorname{bos}}).$$It suffices to show$$\mathbf{\acute{E}t}_{\text{sep}}(S_{\operatorname{bos}})\cong\mathbf{\acute{E}t}_{\text{sep}}(f_{\operatorname{bos}}\colon S_{\operatorname{bos}}^\prime\to S_{\operatorname{bos}}).$$Note that $f_{\operatorname{bos}}\colon S_{\operatorname{bos}}^\prime\to S_{\operatorname{bos}}$ is an integral surjective morphism of algebraic spaces. The result then follows from the ordinary case \cite[Theorem A.2(1a)]{Ryd13}. 
\end{proof}

\subsection{The descent condition for quotients by supergroupoids}\label{sec:descent_condition}

We establish some results on the descent condition for quotients by supergroupoids. As a consequence, we show that integral strongly geometric quotients of algebraic superspaces are categorical. In the ordinary case, these results are due to Rydh in \cite[Section\ 3]{Ryd13}, which we follow closely. We begin with some terminology. Recall from \cite[Definition 3.2]{Ryd13} that a morphism $f\colon (R_W,W)\to(R_X,X)$ of supergroupoids is \textit{square} if the following diagrams
\begin{equation*}
        \begin{tikzcd}
            R_W\arrow[r,"f"]\arrow[d,"s"] & R_X\arrow[d,"s"] & R_W\arrow[r,"f"]\arrow[d,"t"] & R_X\arrow[d,"t"]\\
            W\arrow[r,"f"] & X & W\arrow[r,"f"] & X
        \end{tikzcd}
    \end{equation*}
are Cartesian.
\begin{definition}[{cf.\ \cite[Definition 3.3]{Ryd13}}]\label{def:descent_condition_supergroupoids}
    Let $(R_{X}, X)$ be a supergroupoid and let $q\colon X\to Z_{X}$ be a geometric quotient. Let $P$ be a property of morphisms of algebraic superspaces. We say the quotient $q\colon X\to Z_{X}$ satisfies the descent condition for $\mathbf{\acute{E}t}_{P}$ if for every square, stabilizer-preserving morphism $f\colon (R_{W},W)\to(R_{X},X)$ of supergroupoids such that $(f\colon W\to X)\in\mathbf{\acute{E}t}_{P}(X)$, there exists an algebraic superspace $Z_{W}$ and a Cartesian diagram
    \begin{equation}\label{diag:descent_condition_def}
        \begin{tikzcd}
            W\arrow[r]\arrow[d]\arrow[rd, phantom,"\square"] & X\arrow[d]\\
            Z_{W}\arrow[r] & Z_{X}
        \end{tikzcd}
    \end{equation}
    such that $(Z_{W}\to Z_{X})\in\mathbf{\acute{E}t}_{P}(Z_{X})$. We say that $q\colon X\to Z_{X}$ satisfies the descent condition for $\mathbf{\acute{E}t}_{P}$ uniformly if every flat base change of $q$ satisfies the descent condition for $\mathbf{\acute{E}t}_{P}$.
\end{definition}

Note that the morphism $W\to Z_{W}$ in the diagram above is also a geometric quotient. This is because geometric quotients are compatible with flat base change. The following observation tells us that strongly geometric quotients satisfying the descent conditions are automatically categorical.

\begin{lemma}[{cf.\ \cite[Proposition\ 3.8]{Ryd13}}]\label{lem:geometric_quotients_descent_categorical}
    Let $(R_{X}, X)$ be a supergroupoid and let $q\colon X\to Z_{X}$ be a geometric quotient satisfying the descent condition for $\mathbf{\acute{E}t}_{\operatorname{sep}}$. Then $q$ is a categorical quotient.
\end{lemma}

\begin{proof}
    Let $W$ be an algebraic superspace with an equivariant morphism $r\colon X\to W$. Then we need to produce a unique morphism $f\colon Z_{X}\to W$ such that $r=f\circ q$. First of all, we may assume $W$ is quasi-compact. This is because geometric quotients commute with open immersions. Therefore, we may choose a separated and \'etale presentation $W^\prime\to W$ where $W^\prime$ is an affine superscheme. Let $X^\prime=X\times_WW^\prime$. Note that the projection $X^\prime\to X$ is separated and \'etale by construction. Since $q\colon X\to Z_X$ satisfies the descent condition for $\mathbf{\acute{E}t}_{\operatorname{sep}}$, we get another geometric quotient $q^\prime\colon X^\prime\to Z_{X^\prime}$ with the following Cartesian diagram
    \begin{equation}\label{diag:descent_condition_categorical}
        \begin{tikzcd}
            X^\prime\arrow[r]\arrow[d]\arrow[rd,phantom,"\square"] & X\arrow[d]\\
            Z_{X^\prime}\arrow[r] & Z_{X}.
        \end{tikzcd}
    \end{equation}
    Since $W^\prime$ is affine, a morphism $X^\prime\to W^\prime$ is uniquely determined by a map between the global sections $\Gamma(W^\prime)\to\Gamma(X^\prime)$. Moreover, if the morphism $X^\prime\to W^\prime$ is equivariant, then the map $\Gamma(W^\prime)\to\Gamma(X^\prime)$ factors through the subsuperring of invariants of the $R_{X^\prime}$-action, which is precisely $\Gamma(Z_{X^\prime})$. This in turn induces a unique morphism $f^\prime\colon Z_{X^\prime}\to W^\prime$ such that $r^\prime=f^\prime\circ q^\prime$. By \'etale descent, we obtain a unique morphism $f\colon Z_X\to W$ with the desired property.
\end{proof}

For the rest of this section, we will show that if $q\colon X\to Z_{X}$, as a morphism of superspaces, is of effective descent for separated and \'etale morphisms, then it satisfies the descent condition for $\mathbf{\acute{E}t}_{\operatorname{sep}}$ as a quotient. Let $f\colon(R_{W}, W)\to(R_{X}, X)$ be an \'etale, separated and square morphism of supergroupoids. Let $q\colon X\to Z_{X}$ be a geometric quotient. In this case, an effective descent datum associated to $f\colon W\to X$ is an isomorphism $\varphi\colon W\times_{Z_X}X\to X\times_{Z_X}W$ over $X\times_{Z_X}X$ satisfying the cocycle condition. Consider the following diagram 
\begin{equation}\label{diag:descent_condition}
    \begin{tikzcd}
        & W\times_{Z_X}X\arrow[rd,"f\times id_X"]\arrow[dd,dotted,"\varphi"] & \\
        R_{W}\arrow[ru,"{(s,f\circ t)}"]\arrow[rd,"{(f\circ s, t)}"] & & X\times_{Z_X}X.\\
        & X\times_{Z_X}W\arrow[ru,"id_X\times f"] &
    \end{tikzcd}
\end{equation}

Since $q$ is a geometric quotient, the maps on the left-hand side are both surjective. If $\varphi$ is an isomorphism that fits into the dotted arrow in (\ref{diag:descent_condition}) such that the right-hand side commutes, then its graph $\Gamma_\varphi$ is an open subset of $W\times_{Z_X}W$. This is because $f\colon W\to X$ is \'etale. On the other hand, let $\Gamma_{W/Z_{X}}$ be the image of $j_{W/Z_{X}}=(s,t)\colon R_W\to W\times_{Z_X}W$. The left-hand side of (\ref{diag:descent_condition}) commutes if and only if $\Gamma_{W/Z_{X}}=\Gamma_\varphi$. It is clear that $\Gamma_{W/Z_{X}}\subseteq\Gamma_\varphi$ is equivalent to the commutativity, and the other inclusion follows from the surjectivity of $(s,f\circ t)$. This implies that there is at most one isomorphism $\varphi$ such that (\ref{diag:descent_condition}) commutes. We have essentially shown the following correspondence.

\begin{lemma}\label{lem:quotient_descent_datum_correspondence}
    Let $f\colon(R_{W}, W)\to(R_{X}, X)$ be an \'etale, separated and square morphism of supergroupoids. Let $q\colon X\to Z_{X}$ be a geometric quotient that is of effective $\mathbf{\acute{E}t}_{\operatorname{sep}}$-descent. Then there is a one-to-one correspondence between geometric quotients $W\to Z_W$ with an \'etale separated morphism $Z_W\to Z_X$ such that the diagram (\ref{diag:descent_condition_def}) is Cartesian and effective descent data $\varphi$ making (\ref{diag:descent_condition}) commute. Moreover, there is at most one such geometric quotient $W\to Z_W$ for a fixed effective descent datum $\varphi$.
\end{lemma}

\begin{proof}
    Given an effective descent datum, we get a unique Cartesian square as in (\ref{diag:descent_condition_def}) where $Z_W\to Z_X$ is separated and \'etale since $q$ is of effective $\mathbf{\acute{E}t}_{\operatorname{sep}}$-descent. It follows that $W\to Z_W$ is also a geometric quotient. For the other direction, we start with a geometric quotient $W\to Z_{W}$ with a Cartesian square as in (\ref{diag:descent_condition_def}) where $Z_W\to Z_X$ is separated and \'etale. In this case, we have an isomorphism $W\times_{Z_X}X\cong W\times_{Z_W}W\cong X\times_{Z_X}W$ over $X\times_{Z_X}X$ satisfying the cocycle condition and making (\ref{diag:descent_condition}) commute. The last claim follows from the discussion above.
\end{proof}

\begin{lemma}[{cf.\ \cite[Proposition\ 3.12]{Ryd13}}]\label{lem:super_Rydh_3.12}
    In the situation of \Cref{lem:quotient_descent_datum_correspondence}, assume further that $f$ is stabilizer-preserving. Let $\Gamma_{W/Z_{X}}$ be the image of $j_{W/Z_{X}}=(s,t)\colon R_W\to W\times_{Z_X}W$. Then the following holds:
    \begin{enumerate}[topsep=0pt,noitemsep,label=\normalfont(\arabic*)]
        \item\label{super_Rydh_3.12(1)} the morphisms $f\times_{Z_X}id_{W}\colon W\times_{Z_X}W\to X\times_{Z_X}W$ and $id_{W}\times_{Z_X}f\colon W\times_{Z_X}W\to W\times_{Z_X}X$ are universally bijective when restricted to $\Gamma_{W/Z_{X}}$;
        \item\label{super_Rydh_3.12(2)} the subset $\Gamma_{W/Z_{X}}$ is open in $W\times_{Z_X}W$ if and only if $f\times_{Z_X}id_{W}$ is universally submersive when restricted to $\Gamma_{W/Z_{X}}$; and
        \item\label{super_Rydh_3.12(3)} there exists an isomorphism $\varphi$ making (\ref{diag:descent_condition}) commute if and only if $\Gamma_{W/Z_{X}}$ is open and $\Gamma_\varphi=\Gamma_{W/Z_{X}}$.
    \end{enumerate}
\end{lemma}

\begin{proof}
    We follow the argument given in \cite[Proposition\ 3.12]{Ryd13}. For \ref{super_Rydh_3.12(1)}, it suffices to prove the statement for one of the morphisms. Let $(x,w)\colon\operatorname{Spec}(k)\to X\times_{Z_X}W$ be a point. We need to show that there is exactly one lift of $(x,w)$ to $W\times_{Z_X}W$ that is in the image of $R_W$. Note that the morphism $(f\circ s, t)\colon R_W\to X\times_{Z_X}W$ is surjective, so we get a lift $r\colon\operatorname{Spec}(k)\to R_W$ of $(x,w)$ after replacing $k$ with a field extension of itself. It suffices to show that, after replacing $k$ with further field extensions, every lift of $(x,w)$ to $R_W$ has the same image in $W\times_{Z_X}W$. To this end, let $(w^\prime,w)$ be the image of $r$ in $W\times_{Z_X}W$. Conjugation identifies the fibre over $(w^\prime,w)$ with the stabilizer $\operatorname{Stab}(w^\prime)$. On the other hand, the fibre over $(x,w)$ in $R_{W}$ is isomorphic to $\operatorname{Stab}(x)$. This follows from the assumption that $f$ is square. Since $f$ is also stabilizer-preserving,  we see that the fibre over $(x,w)$ in $R_{W}$ identifies with the fibre over $(w^\prime,w)$ in $R_w$. This proves \ref{super_Rydh_3.12(1)}.

    For \ref{super_Rydh_3.12(2)}, suppose $\Gamma_{W/Z_{X}}$ is open. It follows from \ref{super_Rydh_3.12(1)} that $\Gamma_{W/Z_{X}}\to X\times_{Z_X}W$ is \'etale and universally bijective, and thus an isomorphism. In particular, $f\times\operatorname{id}_W$ is universally submersive. Conversely, if it is universally submersive over $\Gamma_{W/Z_{X}}$, in the sense of \cite[Definition 3.9]{Ryd13}, then it is a universal homeomorphism by \ref{super_Rydh_3.12(1)}. Since $f\times\operatorname{id}_W$ is unramified, we see that $\Gamma_{W/Z_{X}}$ is open by the proof of \cite[Lemma\ 3.10]{Ryd13}.

    For \ref{super_Rydh_3.12(3)}, we know that if such a morphism $\varphi$ exists, then $\Gamma_{W/Z_{X}}$ is given by the graph of $\varphi$, which is open. For the other direction, \ref{super_Rydh_3.12(1)} tells us that $\Gamma_{W/Z_{X}}\to X\times_{Z_X}W$ and $\Gamma_{W/Z_{X}}\to W\times_{Z_X}X$ are universally bijective. If $\Gamma_{W/Z_{X}}$ is open, then these two maps are both isomorphisms as \'etale universally bijective morphisms are isomorphisms. Therefore, we get an induced isomorphism $\varphi$ whose graph is precisely $\Gamma_{W/Z_X}$. It satisfies the cocycle condition by the composition law of the supergroupoid and the uniqueness of the compatible isomorphism in \Cref{lem:quotient_descent_datum_correspondence}.
\end{proof}

\begin{lemma}\label{lem:super_geometric_quotient_descent_condition}
    Let $(R,X)$ be a supergroupoid and let $q\colon X\to Z_{X}$ be a strong geometric quotient. If $q$ is of effective $\mathbf{\acute{E}t}_{\operatorname{sep}}$-descent, then $q$ satisfies the descent condition for $\mathbf{\acute{E}t}_{\operatorname{sep}}$ in the sense of \Cref{def:descent_condition_supergroupoids}.
\end{lemma}

\begin{proof}
    Let $f\colon W\to X$ be a square, separated, \'etale, and stabilizer-preserving morphism. \Cref{lem:quotient_descent_datum_correspondence} tells us that it suffices to show that $\Gamma_{W/Z_{X}}$ is open. Consider the following Cartesian diagram.
    \begin{equation*}
        \begin{tikzcd}[column sep=huge]
            R_W\arrow[r,"(f\times\operatorname{id}_{W})\circ j_{W/Z_X}"]\arrow[d]\arrow[rd,phantom,"\square"] & X\times_{Z_X}W\arrow[d]\\
            R_X\arrow[r,"j_{X/Z_X}"] & X\times_{Z_X}X.
        \end{tikzcd}
    \end{equation*}
    Since $q$ is a strongly geometric quotient, the map $j_{X/Z_X}\colon R_{X}\to X\times_{Z_X}X$ is universally submersive. It follows that $(f\times\operatorname{id}_{W})\circ j_{W/Z_X}$ is also universally submersive. In other words, the morphism $(f\times\operatorname{id}_{W})$ is universally submersive over $\Gamma_{W/Z_X}$. Hence $\Gamma_{W/Z_X}$ is open by \Cref{lem:super_Rydh_3.12}\ref{super_Rydh_3.12(2)}. 
\end{proof}

\begin{theorem}\label{thm:super_geometric_quotient_descent_categorical}
    Let $(R,X)$ be a supergroupoid and let $q\colon X\to Z_{X}$ be a strong geometric quotient. If $q$ is integral, then it satisfies the descent condition for $\mathbf{\acute{E}t}_{\text{sep}}$ uniformly. In particular, $q\colon X\to Z_{X}$ is a categorical quotient.
\end{theorem}

\begin{proof}
    Since $q\colon X\to Z_{X}$ is a topological quotient, it is surjective. Since $q$ is integral and surjective, it remains so after arbitrary base change. By \Cref{thm:effective_descent_integral_sep}, $q$ is of universal effective $\mathbf{\acute{E}t}_{\text{sep}}$-descent. By \Cref{lem:super_geometric_quotient_descent_condition} and \Cref{prop:geometric_quotient_uniform_local}, $q$ satisfies the descent condition for $\mathbf{\acute{E}t}_{\text{sep}}$ uniformly. It follows from \Cref{lem:geometric_quotients_descent_categorical} that $q$ is a categorical quotient.
\end{proof}

Recall that an equivariant morphism is a \textit{GC quotient} if it is a strong geometric quotient and satisfies the descent condition for separated \'etale morphisms uniformly. \Cref{thm:super_geometric_quotient_descent_categorical} implies that a strong geometric quotient that is integral is a GC quotient. We end this section with the following gluing result for geometric quotients, which is essential in the proof of the main theorem of this article. 

\begin{theorem}[{cf.\ \cite[Theorem\ 3.19]{Ryd13}}]\label{thm:super_Kollar_Rydh}
    Let $S$ be a superscheme. Let $f\colon(R_W,W)\to(R_X,X)$ be a separated, \'etale, square, stabilizer-preserving, and surjective morphism of supergroupoids. Let $Q=W\times_{X}W$. If $(R_W,W)$ admits a strong geometric quotient $q\colon W\to Z_W$ that is integral, then there exist GC quotients $X\to Z_X$ and $Q\to Z_Q$ such that $Z_X$ is the quotient of \'etale equivalence relation $Z_Q\rightrightarrows Z_W$ and the squares in the following diagram
    \begin{equation}\label{dia:super_Kollar_Rydh}
        \begin{tikzcd}
            Q\arrow[r, shift left=1]\arrow[r, shift right=1]\arrow[d] & W\arrow[r]\arrow[d] & X\arrow[d]\\
            Z_Q\arrow[r, shift left=1]\arrow[r, shift right=1] & Z_W\arrow[r] & Z_{X}
        \end{tikzcd}
    \end{equation}
    are Cartesian. Moreover, the quotient $X\to Z_X$ is also categorical.
\end{theorem}

\begin{proof}
    Let $p_1,p_2\colon Q\to W$ be the two projections. Since $f$ is separated, \'etale, square, and stabilizer-preserving, so are $p_1$ and $p_2$. By \Cref{thm:super_geometric_quotient_descent_categorical}, we have a GC quotient $W\to Z_W$. In particular, it satisfies the descent condition. Applying this to the two projections $p_1,p_2$ yield two quotients $Q\to Z_{Q}$ and $Q\to Z_{Q}^\prime$ corresponding to $p_1$ and $p_2$ such that the two squares
    \begin{equation*}
        \begin{tikzcd}
            Q\arrow[r,"p_1"]\arrow[d] & W\arrow[d] & Q\arrow[r,"p_2"]\arrow[d] & W\arrow[d]\\
            Z_{Q}\arrow[r] & Z_W & Z_Q^{\prime}\arrow[r] & Z_W
        \end{tikzcd}
    \end{equation*}
    are Cartesian. It follows that $Q\to Z_{Q}$ and $Q\to Z_{Q}^\prime$ are integral and surjective strong geometric quotients. Applying \Cref{thm:super_geometric_quotient_descent_categorical} tells us that they are also GC quotients and thus categorical. By the universal property of categorical quotients, we see that $Z_{Q}\cong Z_{Q}^\prime$. We claim that this defines a supergroupoid over $S$. Note that we have stabilizer-preserving morphisms $(p_1,p_2)\colon Q\rightrightarrows W$. Consider the diagonal map $W\to Q$, the morphism $Q\to Q$ given by exchanging the factors, and the projection $Q\times_{W}Q\cong W\times_X W\times_X W\to Q\cong W\times_X W$ to the first and the last factors. These maps induce square, separated, \'etale, and stabilizer-preserving morphisms between the corresponding supergroupoids. By the descent condition and uniqueness of geometric quotients, these maps descend to separated and \'etale morphisms $Z_W\to Z_Q$, $Z_{Q}\to Z_{Q}$, and $Z_Q\times_{Z_W}Z_Q\to Z_Q$. This gives us a supergroupoid $ Z_Q\rightrightarrows Z_W$ as claimed.

    We now show that the maps $Z_Q\rightrightarrows Z_W$ form an \'etale equivalence relation. We do so by showing the map $Z_Q\to Z_W\times Z_W$ is a monomorphism, or equivalently, the map $Z_Q\times_{Z_W\times Z_W}Z_W\to Z_W$ is an isomorphism. Consider the morphism $W\times_{Z_W\times Z_W}Z_W\to Q\times_{Z_W\times Z_W}Z_W$ given by the diagonal $W\to Q$ and the identity map on $Z_W$. This morphism is an open immersion since $f$ is \'etale. Let $(w,w^\prime)$ be a $k$-point of $Q=W\times_{X} W$ such that $q(w)=q(w^\prime)$. By taking a field extension of $k$, we may assume that $(w,w^\prime)$ lifts to a $k$-point $r$ of $R_W$. Since $f(w)=f(w^\prime)$, we see that the image of $r$ along $f$ lies in $\operatorname{Stab}_X(f(w))$. Since $f$ is stabilizer-preserving, we have $\operatorname{Stab}_W(w)\cong\operatorname{Stab}_X(f(w))$. There exists $g\in\operatorname{Stab}_W(w)$ such that $f_R(g)=f_R(r)$. Note that $g$ and $r$ have the same source and image in $R_X$. Since the source square is Cartesian, we must have $g=r$. It follows that $w^\prime=t(r)=t(g)=w$. It follows that $(w,w^\prime)$ lifts to a $k$-point of $W$. This implies that the map $W\times_{Z_W\times Z_W}Z_W\to Q\times_{Z_W\times Z_W}Z_W$ is surjective, and thus an isomorphism. Since $W\to Z_W$ is of effective $\mathbf{\acute{E}t}_{\operatorname{sep}}$-descent, we get an isomorphism $Z_Q\times_{Z_W\times Z_W}Z_W\cong Z_W$ as desired.

    Let $Z_X$ be the quotient of the \'etale equivalence relation $Z_Q\rightrightarrows Z_W$. By construction, $Z_X$ is an algebraic superspace. By construction, the composition $W\to Z_W\to Z_X$ is equivariant under $Q\rightrightarrows W$. Since $X$ is the quotient of $W$ by $Q$, this composition descends to a morphism $X\to Z_X$ and \'etale descent gives the following Cartesian diagram
    \begin{equation*}
        \begin{tikzcd}
            W\arrow[r,"f"]\arrow[d] & X\arrow[d]\\
            Z_{W}\arrow[r] & Z_X.
        \end{tikzcd}
    \end{equation*}
    Since $W\to Z_W$ is integral and surjective, so is $X\to Z_{X}$ by \Cref{lem:integral_morphism_etale_local}. By \Cref{prop:geometric_quotient_uniform_local}, we see that $X\to Z_X$ is also a strong geometric quotient. By \Cref{thm:super_geometric_quotient_descent_categorical}, $X\to Z_X$ is a GC quotient that is also categorical. This completes the proof. 
\end{proof}

\section{The Keel--Mori theorem for Deligne--Mumford superstacks}\label{sec:keel_mori}

\subsection{Quotients by finite \'etale supergroupoids}\label{sec:quotient}

We first establish the existence of coarse moduli superspaces for algebraic superstacks that admit finite, \'etale coverings of superschemes. Our argument is similar to those for the ordinary case in \cite{Con05} and \cite[Section\ 6.2]{Ols16}, which we follow closely. The main effort is devoted to proving the affine case, from which the general case follows using a reduction argument. The restriction to finite \'etale supergroupoids is necessary for us. 

Consider a supergroupoid $(U,R,s,t,c)$ over $S$ that is finite locally free, and $U$ is affine. It follows that $R$ is also affine. Let $U=\operatorname{sSpec}(A)$ and $R=\operatorname{sSpec}(B)$ for some superrings $A$ and $B$. By assumption, $B$ is a finitely generated $A$-module. With a slight abuse of notation, we write $s,t\colon A\rightrightarrows B$ for the corresponding morphisms of superrings. Set $C=B\otimes_{t,A,s}B$, and we obtain a map $c\colon B\to C$ induced by the composition map of our supergroupoid. Let $A^{R}$ be the equalizer of $s,t\colon A\rightrightarrows B$. The following is a variant of \cite[Lemma\ 6.2.9]{Ols16} in the super setting.

\begin{lemma}\label{lem:invariant_superring}
    Let $a$ be an element in $A$. If $s,t$ are \'etale, then there exists a monic polynomial $P_{a}(x)\in A_{0}^{R}[x]$ such that $P_{a}(a)=0$. In particular, the extension $A^R\to A$ is integral.
\end{lemma}

\begin{proof}
    We apply the bosonic quotient functor to the finite \'etale supergroupoid $(U,R,s,t,c)$. This gives us a finite \'etale groupoid $(U_{\operatorname{ev}},R_{\operatorname{ev}},s_{\operatorname{ev}},t_{\operatorname{ev}},c_{\operatorname{ev}})$ where $U_{\operatorname{ev}}=\operatorname{Spec}(A_0)$ and $R_{\operatorname{ev}}=\operatorname{Spec}(B_0)$. This is because taking the bosonic quotient commutes with \'etale base change \Cref{lem:ev_commutes_etale_base_change}. By the ordinary case \cite[Lemma\ 6.2.9]{Ols16}, we see that $A_0$ is integral over $(A_0)^{R_0}=\operatorname{Eq}(s_0,t_0\colon A_0\to B_0)$. But by the universal property of equalizers, we have $(A^R)_0=(A_0)^{R_0}$. The statement then follows verbatim from \Cref{lem:integral_bosonic}.
\end{proof}

\Cref{lem:invariant_superring} is sufficient for this article. The assumption that the supergroupoid is \'etale is sufficient for the bosonic quotient $(U_{\operatorname{ev}},R_{\operatorname{ev}},s_{\operatorname{ev}},t_{\operatorname{ev}},c_{\operatorname{ev}})$ to be an ordinary groupoid. Without the \'etale hypothesis, taking bosonic quotients need not commute with fibre products. Therefore the argument above does not produce an ordinary groupoid in general. We demonstrate how this could fail when the groupoid is no longer \'etale.

\begin{example}\label{ex:G_a^-}
    Let $k$ be a field such that $\operatorname{char}(k)\neq 2$. Consider the supergroup $\mathbb{G}_a^{0|1}=\operatorname{sSpec}(k[\theta])$ where $\theta$ is an odd variable. Note that the map $\mathbb{G}_a^{0|1}$ is finite locally free but not \'etale over $k$. There is a natural trivial action of $\mathbb{G}_a^{0|1}$ on $\operatorname{Spec}(k)$. The source and target maps are both the projection map $\mathbb{G}_a^{0|1}\to\operatorname{Spec}(k)$ and the composition map is precisely the multiplication map $\mathbb{G}_a^{0|1}\times_k\mathbb{G}_a^{0|1}\to \mathbb{G}_a^{0|1}$. This gives us a supergroupoid $(\operatorname{Spec}(k),\mathbb{G}_a^{0|1},s,t,c)$. We claim that the bosonic quotient of this supergroupoid is not an ordinary groupoid. Indeed, taking the bosonic quotient on objects yields $s_{\operatorname{ev}},t_{\operatorname{ev}}\colon\operatorname{Spec}(k)\rightrightarrows\operatorname{Spec}(k)$, where $s_{\operatorname{ev}}=t_{\operatorname{ev}}$ is the identity map. However, the bosonic quotient of the composition map is given by $c_{\operatorname{ev}}\colon\operatorname{Spec}(k[\epsilon]/(\epsilon^2))\to\operatorname{Spec}(k)$. So we see that$$(R\times_UR)_{\operatorname{ev}}\cong\operatorname{Spec}(k[\epsilon]/(\epsilon^2))\ncong\operatorname{Spec}(k)\cong R_{\operatorname{ev}}\times_{U_{\operatorname{ev}}}R_{\operatorname{ev}}.$$This shows that the bosonic quotient of the composition morphism need not have the required source to satisfy the ordinary groupoid composition law.
\end{example}

However, \Cref{lem:invariant_superring} is false for finite locally free supergroupoids. To this end, we provide the following counterexample.

\begin{example}\label{ex:G_a^-action_not_integral}
    Consider the supergroup $\mathbb{G}_a^{0|1}=\operatorname{sSpec}(\mathbb{C}[\theta])$ acting on $\mathbb{A}^{1|1}=\operatorname{sSpec}(\mathbb{C}[x| \xi])$ by $\theta\cdot(x,\xi)=(x+\theta\xi,\xi)$. We claim that $\mathbb{C}[x|\xi]^{\mathbb{G}_a^{0|1}}=\mathbb{C}\oplus\xi \mathbb{C}[x]$. To see this, we first observe that every element in $\mathbb{C}[x|\xi]$ is of the form $f(x)+\xi g(x)$ for some $f(x),g(x)\in\mathbb{C}[x]$. We see that an element $f(x)+\xi g(x)\in\mathbb{C}[x|\xi]$ is $\mathbb{G}_a^{0|1}$-invariant if and only if $f(x)+\xi g(x)=f(x+\theta\xi)+\xi g(x+\theta\xi)$. Evaluating the right-hand side gives
    \begin{align*}
        f(x)+\xi g(x)&=f(x+\theta\xi)+\xi g(x+\theta\xi)\\
        &=f(x)+\theta\xi f^\prime(x)+\xi(g(x)+\theta\xi g^\prime(x))\\
        &=f(x)+\theta\xi f^\prime(x)+\xi g(x)-\theta\xi^2 g^\prime(x)\\
        &=f(x)+\theta\xi f^\prime(x)+\xi g(x).
    \end{align*}
    This implies that $f(x)+\xi g(x)\in\mathbb{C}[x|\xi]^{\mathbb{G}_a^{0|1}}$ if and only if $f(x)$ is constant, and the claim follows. It is clear that $\mathbb{C}\oplus \xi\mathbb{C}[x]\to\mathbb{C}[x|\xi]$ cannot be an integral extension. This follows from \Cref{lem:integral_bosonic} and the fact that $\mathbb{C}\to \mathbb{C}[x]$ is not an integral extension.
\end{example}

Let $X=\operatorname{sSpec}(A^{R})$ and let $q\colon U\to X$ be the morphism of superschemes induced by the inclusion $A^{R}\to A$. Let $\mathcal{X}$ be the quotient superstack associated to our supergroupoid $s,t\colon R\rightrightarrows U$. Then we obtain an induced map $\pi\colon\mathcal{X}\to\operatorname{sSpec}(A^{R})$ making the following diagram commute 
\begin{equation}
    \begin{tikzcd}
        U=\operatorname{sSpec}(A)\arrow[d,"p"]\arrow[rd,"q"] & \\
        \mathcal{X}=[U/R]\arrow[r,"\pi"] & X=\operatorname{sSpec}(A^{R}),
    \end{tikzcd}
\end{equation}
where $p$ is the quotient map, which is finite and flat. By construction, we have the following exact sequence
\begin{equation}\label{dia:invariants}
    \begin{tikzcd}
        0\arrow[r] & A^{R}\arrow[r,hookrightarrow] & A\arrow[r, shift left,"s"]\arrow[r,shift right, swap, "t"] & B.
    \end{tikzcd}
\end{equation}
The following lemma implies that the formation of $\pi$ is compatible with flat base change.

\begin{lemma}\label{lem:flat_base_change}
    Let $A^{R}\to D$ be a flat morphism of superrings. Let $s^\prime,t^\prime\colon R^\prime=\operatorname{sSpec}(B\otimes_{A^{R}}D)\rightrightarrows U^\prime=\operatorname{sSpec}(A\otimes_{A^{R}}D)$ be the supergroupoid obtained by base change to $D$. Let $(A\otimes_{A^{R}}D)^{R^\prime}$ be the equalizer of $s^\prime$ and $t^\prime$. Then we have $D=(A\otimes_{A^{R}}D)^{R^\prime}$.
\end{lemma}

\begin{proof}
    Since $D$ is flat over $A^{R}$, we obtain the following exact sequence from \ref{dia:invariants}
    \begin{equation}
        \begin{tikzcd}
            0\arrow[r] & D\arrow[r, hookrightarrow] & A\otimes_{A^{R}}D\arrow[r, shift left, "s^\prime"]\arrow[r,shift right, swap, "t^\prime"] & B\otimes_{A^{R}}D.
        \end{tikzcd}
    \end{equation}
    This tells us that $D$ is equal to the equalizer of $s^\prime$ and $t^\prime$, which is $(A\otimes_{A^{R}}D)^{R^\prime}$. 
\end{proof}

We are now ready to prove the main result of this section, which says that quotients by finite \'etale supergroupoids of affine superschemes give us strongly geometric quotients.

\begin{theorem}\label{thm:superspace_finite_flat_covering}
    Let $D$ be a superring. Let $s,t\colon R\rightrightarrows U$ be a finite \'etale supergroupoid over $\operatorname{sSpec}(D)$. If $U$ is an affine superscheme, then it admits a strong geometric quotient $q\colon U\to X$. Moreover, the quotient $q\colon U\to X$ satisfies the descent condition for $\mathbf{\acute{E}t}_{\text{sep}}$, and is thus, categorical. In particular, $q$ is integral and surjective. If $U$ is of finite type over $\operatorname{sSpec}(D)$, then $q$ is finite. If $D$ is also Noetherian, then $X$ is also of finite type.
\end{theorem}

\begin{proof}
    Let $U=\operatorname{sSpec}(A)$ and $R=\operatorname{sSpec}(B)$ as before. Set $X=\operatorname{sSpec}(A^{R})$. We claim that the natural map $q\colon U\to X$ is a strongly geometric quotient. To see it is a strong topological quotient, we readily reduce to the ordinary case by taking the bosonic quotient. Indeed, we get an ordinary groupoid $s_{\operatorname{ev}},t_{\operatorname{ev}}\colon R_{\operatorname{ev}}\rightrightarrows U_{\operatorname{ev}}$ with the same underlying topological space. By construction, this groupoid admits a quotient given by $q_{\operatorname{ev}}\colon U_{\operatorname{ev}}\to X_{\operatorname{ev}}$. Since $q$ is integral and surjective, it is universally submersive. Moreover, the fibres of $q$ are precisely the $R$-orbits.
    
    Consider the map $j\colon R\to U\times_X U$. We claim this map is surjective. To see this, let $(u_1,u_2)\colon\operatorname{Spec}(k)\to U\times_X U$ be geometric points with the same image in $X$. It follows that their images in $U_{\operatorname{ev}}$ map to the same point in $X_{\operatorname{ev}}$. Replacing $k$ by a field extension if necessary, we may find a geometric point $r_0$ such that $s_{\operatorname{ev}}(r_0)=(u_1)_{\operatorname{ev}}$ and $t_{\operatorname{ev}}(r_0)=(u_2)_{\operatorname{ev}}$. Since $|R_{\operatorname{ev}}|=|R|$, we may lift $r_0$ uniquely to a geometric point $r$ of $R$ such that $s(r)=u_1$ and $t(r)=u_2$. This proves the claim. Since $R\to U$ is finite and $U\times_XU$ is affine, we see that $j$ is also finite, and thus universally submersive. This shows that $q$ is a strong topological quotient. By \Cref{lem:flat_base_change}, we see that $X$ is also a geometric quotient, and thus a strongly geometric quotient as desired. Moreover, we also know that $q\colon U\to X$ is integral and surjective by \Cref{lem:invariant_superring}. Then \Cref{thm:super_geometric_quotient_descent_categorical} applies, and we see that $q$ satisfies the descent condition for $\mathbf{\acute{E}t}_{\text{sep}}$, and is categorical. We know that $A^{R}\to A$ is an integral extension. It follows that $q$ is integral and surjective. 
    
    Suppose $U$ is of finite type over a superring $D$. Consider the superring homomorphisms $D\to A^R\to A$. We know that $A^{R}\to A$ is an integral extension. Since $A$ is finitely generated over $D$, it is also finitely generated over $A^R$. By \Cref{lem:integral+finite_type=finite}, we see that $A^R\to A$ is finite and thus $q$ is finite. Suppose that $D$ is also Noetherian. By replacing $D$ by its image if necessary, we may assume $D$ is a subsuperring of $A^R$. By the Artin--Tate lemma (\Cref{lem:super_artin_tate}), $A^R$ is finitely generated over $D$ and the result follows.
\end{proof}

An interesting instance of \Cref{thm:superspace_finite_flat_covering} is the quotient of an affine superscheme by a finite \'etale group superscheme. Let $G$ be a finite \'etale group superscheme acting on a superring $A$. Then \Cref{thm:superspace_finite_flat_covering} implies that there is strongly geometric quotient $q\colon\operatorname{sSpec}(A)\to\operatorname{sSpec}(A^{G})$, where $A^{G}\subseteq A$ is the subsuperring of invariants. In particular, if $A$ is a finitely generated superalgebra over a Noetherian ring, then so is $A^{G}$. In the ordinary case, one shows that $A^{G}\to A$ is integral using the monic polynomial $\prod_{g\in G}(x-g(a))$ for every element $a\in A$. The same argument readily extends to finite group actions on superrings using \Cref{lem:integral_bosonic}. However, the situation become complicated for finite supergroup actions. This prompts us to use the argument in \Cref{lem:invariant_superring} instead.

We now establish a general version of \Cref{thm:superspace_finite_flat_covering}, which is used in the proof of the main theorem. We say a superscheme $X$ is AF if every finite set of points of $|X|$ is contained inside an affine open subsuperscheme of $X$. For example, affine superschemes, quasi-affine superschemes, and disjoint unions of affine superschemes are AF.

\begin{lemma}\label{lem:AF_superschemes}
    Let $s,t\colon R\rightrightarrows U$ be a finite \'etale supergroupoid. If $U$ is AF, then the orbit $O(u)\subseteq U$ is contained in an $R$-invariant affine open subset of $U$ for every $u\in|U|$.
\end{lemma}

\begin{proof}
    Let $u\in|U|$. Consider the bosonic quotient $s_{\operatorname{ev}},t_{\operatorname{ev}}\colon R_{\operatorname{ev}}\rightrightarrows U_{\operatorname{ev}}$. Note that $|U|=|U_{\operatorname{ev}}|$ and $|R|=|R_{\operatorname{ev}}|$. Applying \cite[Lemma 4.3]{Ryd13} to $s_{\operatorname{ev}},t_{\operatorname{ev}}\colon R_{\operatorname{ev}}\rightrightarrows U_{\operatorname{ev}}$ yields an $R_{\operatorname{ev}}$-invariant affine open subset $W_{\operatorname{ev}}\subseteq U_{\operatorname{ev}}$ containing the orbit of $u$ in $U_{\operatorname{ev}}$. Let $W=W_{\operatorname{ev}}\times_{U_{\operatorname{ev}}}U$. We see that $W\subseteq U$ is an affine open subsuperscheme containing $O(u)$. It remains to show that $W$ is $R$-invariant. We claim that $s^{-1}(|W|)=t^{-1}(|W|)$ in $|R|$. Indeed, this follows from the fact that $|W|=|W_{\operatorname{ev}}|$ and $W_{\operatorname{ev}}$ is $R_{\operatorname{ev}}$-invariant. Therefore, $W$ is $R$-invariant and the result follows.
\end{proof}

\begin{theorem}\label{thm:quotients_by_finite_flat_supergroupoid}
    Let $s,t\colon R\rightrightarrows U$ be a finite \'etale supergroupoid over a superscheme $S$. If $U$ is an AF superscheme, then it admits a strong geometric quotient $q\colon U\to X$ that is categorical. In particular, $q$ is integral and surjective. If $U$ is locally of finite type over $S$, then $q$ is finite. If $S$ is also locally Noetherian, then $X$ is also locally of finite type over $S$.
\end{theorem}

\begin{proof}
    Let $u\in |U|$. By \Cref{lem:AF_superschemes}, there exists an $R$-invariant affine open subsuperscheme $W_u$ containing $O(u)$. Therefore we obtain another supergroupoid $s,t\colon R|_{W_u}\rightrightarrows W_u$. By \Cref{thm:superspace_finite_flat_covering}, the restricted supergroupoid admits a strong geometric quotient $q_u\colon W_u\to X_u$ that is also categorical. Since this is true for every $u\in|U|$, we obtain an affine $R$-invariant open cover $\{W_{u}\}_{u\in|u|}$ of $U$. We claim that $\{q_u\colon W_u\to X_u\}_{u\in|u|}$ glue to a strongly geometric quotient $q\colon U\to X$ that is categorical. Indeed, let $W_{uv}=W_u\cap W_v$ for every pair of distinct points $u,v$ of $U$. Note that $U$ is AF and thus separated. It follows that $W_{uv}$ is affine. We may further restrict the supergroupoid $s,t\colon R|_{W_u}\rightrightarrows W_u$ to $s,t\colon R|_{W_{uv}}\rightrightarrows W_{uv}$. By \Cref{thm:superspace_finite_flat_covering}, we get a strong geometric quotient $q_{uv}\colon W_{uv}\to X_{uv}$ that is categorical. Since $W_{uv}\subseteq W_u$ is $R$-invariant and $q_u$ is a topological quotient, its image is open in $X_u$. By universal property, we see that $q_u(W_{uv})=X_{uv}$. Applying the same argument to $W_{vu}$ also gives a strong geometric quotient $q_{vu}\colon W_{vu}\to X_{vu}$ that is categorical, and $q_{v}$ is a topological quotient; we see that $q_v(W_{vu})=X_{vu}$ is open. Apply the universal property one more time yields $X_{uv}\cong X_{vu}$. Since this is true for every pair of points $u,v$ of $U$, the claim follows from a standard glueing argument. The rest of the statements can be checked on the affine open cover $\{X_u\}$ and thus follow from \Cref{thm:superspace_finite_flat_covering}.
\end{proof}

\subsection{Stabilizer-preserving morphisms of superstacks}\label{sec:stab}

Let $\mathcal{X}$ be an algebraic superstack with finite inertia. In this subsection, we produce a representable, \'etale, separated, surjective, and stabilizer-preserving morphism $\mathcal{W}\to\mathcal{X}$ such that $\mathcal{W}$ admits a finite flat presentation from an AF superscheme. The following is a variant of \cite[Proposition\ 6.11]{Ryd13} for superstacks. We say an algebraic superstack $\mathcal{X}$ has a quasi-finite separated presentation if there exists a locally quasi-finite, separated, representable, surjective morphism of superstacks $U\to\mathcal{X}$ where $U$ is a superscheme.

\begin{lemma}\label{lem:Nisnevich_covering_super}
    Let $\mathcal{X}$ be a Deligne--Mumford superstack with separated diagonal. Then there exist morphisms of algebraic superstacks$$V\overset{v}{\longrightarrow}\mathcal{W}\overset{p}{\longrightarrow}\mathcal{X}$$such that $V$ is an AF superscheme; $v$ is finite, \'etale, surjective, and locally of finite presentation; and $p$ is representable, \'etale, and separated.
\end{lemma}

\begin{proof}
    By assumption, we have a separated and \'etale presentation $U\to\mathcal{X}$ where $U$ is a superscheme. By taking a Zariski open covering, we may assume that $U$ is a disjoint union of affine superschemes, which is AF. Since $U\to\mathcal{X}$ is separated and \'etale, we obtain a morphism of algebraic superstacks $\mathcal{H}=\mathcal{SH}\textit{ilb}_{U/\mathcal{X}}\to\mathcal{X}$ that is representable, separated, and \'etale by \Cref{cor:SHilb_quasi_finite}. Since $U\to\mathcal{X}$ is separated and \'etale, a morphism $T\to\mathcal{H}$ corresponds to a closed immersion $i\colon Z\hookrightarrow U\times_{\mathcal{X}}T$ such that the composition $c\colon Z\hookrightarrow U\times_{\mathcal{X}}T\to T$ is finite and \'etale. In particular, $ Z\hookrightarrow U\times_{\mathcal{X}}T$ is an open and closed immersion.
    
    Now let $\mathcal{W}=\mathcal{H}\setminus \mathcal{SH}\textit{ilb}_{U/\mathcal{X}}^{0|0}$. It follows that the restriction $p\colon\mathcal{W}\to\mathcal{X}$ is \'etale. Let $V\to\mathcal{W}$ be the universal family. By construction, the map $v\colon V\to\mathcal{W}$ is finite, locally free of positive rank on all connected components, and thus surjective. Then the composition $v\colon V\hookrightarrow U\times_{\mathcal{X}}\mathcal{W}\to\mathcal{W}$ is finite, \'etale, and finitely presented. Moreover, the composition $\mathcal{W}\to \mathcal{X}$ is \'etale, representable, and separated. It follows that $U\times_\mathcal{X}\mathcal{W}\to U$ is also \'etale, representable, and separated. Since $U$ is AF, we see that $U\times_\mathcal{X}\mathcal{W}$ is AF \Cref{prop:superscheme_AF}. Therefore, $V$ is AF as it is an open and closed subsuperscheme of an AF superscheme. This completes the proof.
\end{proof}

The key ingredient in the proof of \Cref{lem:Nisnevich_covering_super} is the representability of the relative super Hilbert functor for separated \'etale morphisms of superschemes. See \cite[Definition\ 4.1]{BHP23} for the precise definition of the super Hilbert functor. In \Cref{sec:SHilb}, we establish a representability result sufficient for \Cref{lem:Nisnevich_covering_super}. The argument relies on the super Hilbert superspace of even points. In the ordinary case, this is due to Rydh in \cite{Ryd11a}. In the Noetherian case, one may alternatively apply \cite[Theorem\ 4.3]{BHP23} to obtain a much more general representability result of the relative super Hilbert functor for every quasi-affine morphism of superstacks.

We now introduce the following notion of stabilizer-preserving morphisms for superstacks. It is a natural extension of stabilizer-preserving morphisms for ordinary stacks.

\begin{definition}\label{def:stab_preserving}
    A morphism $f\colon\mathcal{X}\to\mathcal{Y}$ of algebraic superstacks is \textit{stabilizer-preserving} if for every point $x$ of $\mathcal{X}$, the induced morphism $G_{x}\cong I_{\mathcal{X}}\times_{\mathcal{X}}\operatorname{Spec}k\to G_{f(x)}\cong I_{\mathcal{Y}}\times_{\mathcal{Y}}\mathcal{X}\times_{\mathcal{X}}\operatorname{Spec}k$ on the stabilizers is an isomorphism for some representative $x\colon\operatorname{Spec}k\to\mathcal{X}$.
\end{definition}

Stabilizer-preserving morphisms of superstacks are closely related to stabilizer-preserving morphisms of supergroupoids. Let $f\colon \mathcal{X}\to\mathcal{Y}$ be a representable morphism of algebraic superstacks. Choose a smooth presentation $U_\mathcal{Y}\to\mathcal{Y}$. Set $U_\mathcal{X}=U_\mathcal{Y}\times_Y\mathcal{X}$, $R_\mathcal{X}=U_\mathcal{X}\times_\mathcal{X}U_\mathcal{X}$ and $R_\mathcal{Y}=U_\mathcal{Y}\times_\mathcal{Y}U_\mathcal{Y}$. Consider the supergroupoids $(R_\mathcal{X}, U_\mathcal{X})$ and $(R_\mathcal{Y}, U_\mathcal{Y})$. The representable morphism $f$ then induces a square morphism $(R_\mathcal{X}, U_\mathcal{X})\to(R_\mathcal{Y}, U_\mathcal{Y})$ of supergroupoids. In this case, the morphism $f$ is stabilizer-preserving if and only if $(R_\mathcal{X}, U_\mathcal{X})\to(R_\mathcal{Y}, U_\mathcal{Y})$ is so as morphism of supergroupoids. 

\begin{proposition}\label{prop:stabilizer_preserving}
    In this situation of \Cref{lem:Nisnevich_covering_super}, if $\mathcal{X}$ has finite inertia, then it can be arranged that $p$ is stabilizer-preserving and surjective.
\end{proposition}

\begin{proof}
We claim that the locus in $\mathcal{W}$ where $p$ is stabilizer-preserving is open. Let $\mathcal{I}_{\mathcal{X}}$ and $\mathcal{I}_{\mathcal{W}}$ be the inertia stacks of $\mathcal{X}$ and $\mathcal{W}$ respectively. Since $\mathcal{I}_{\mathcal{X}}\to\mathcal{X}$ is finite, so is the base change $p_{2}\colon\mathcal{I}_{\mathcal{X}}\times_{\mathcal{X}}\mathcal{W}\to\mathcal{W}$. The universal property of fibre products gives us a morphism $\mathcal{I}_{\mathcal{W}}\to\mathcal{I}_{\mathcal{X}}\times_{\mathcal{X}}\mathcal{W}$. The fibre of this morphism at a point $z\in|\mathcal{W}|$ is precisely the induced map $G_{z}\to G_{p(z)}$ on the stabilizers. Consider the following Cartesian diagram
    \begin{equation}
        \begin{tikzcd}
            \mathcal{I}_{\mathcal{W}}\arrow[r]\arrow[d] & \mathcal{I}_{\mathcal{X}}\times_{\mathcal{X}}\mathcal{W}\arrow[d]\\
            \mathcal{W}\arrow[r] & \mathcal{W}\times_{\mathcal{X}}\mathcal{W}.
        \end{tikzcd}
    \end{equation}
    Since $p$ is representable, \'etale, and separated, its diagonal is an open and closed immersion. It follows that $\mathcal{I}_{\mathcal{W}}$ is an open and closed subsuperstack of $\mathcal{I}_{\mathcal{X}}\times_{\mathcal{X}}\mathcal{W}$. Set $\mathcal{Z}=|\mathcal{I}_{\mathcal{X}}\times_{\mathcal{X}}\mathcal{W}|\setminus|\mathcal{I}_{\mathcal{W}}|$. Since  $p_{2}\colon\mathcal{I}_{\mathcal{X}}\times_{\mathcal{X}}\mathcal{W}\to\mathcal{W}$ is finite, we see that $p_{2}(\mathcal{Z})$ is closed. It follows that locus in $\mathcal{W}$ where $p$ is stabilizer-preserving, which is precisely the complement of $p_{2}(\mathcal{Z})$, is open in $\mathcal{W}$. Therefore we may replace $\mathcal{W}$ by $\mathcal{W}\setminus p_{2}(\mathcal{Z})$.

    Let $x\colon\operatorname{Spec}k\to\mathcal{X}$ where $k$ is algebraically closed. We now show that $x$ lifts to a point in the stabilizer-preserving locus of $p$ in $|\mathcal{W}|$. Note that the stabilizer $G_{x}$ acts on the fibre $U_{x}$. By assumption, $U_{x}$ is discrete and $G_{x}$ is finite. It follows that the orbit $Z$ of any point of $|U_{x}|$ is an open subscheme of $U_{x}$ that is invariant under the action of $G_{x}$ and finite over $k$. By construction, the isomorphism class of $Z$ corresponds to a $k$-point of $\mathcal{W}$ that is in the stabilizer-preserving locus of $p$, and the result follows.
\end{proof}

Note that \Cref{prop:stabilizer_preserving} is the reason why the finiteness assumption on the inertia stack is necessary in the main theorem.

\subsection{Coarse moduli superspaces}\label{sec:cmss}

We are now ready to prove \Cref{thm:main}. We first recall the definition of coarse moduli spaces.

\begin{definition}\label{def:super_cms}
    Let $\pi\colon\mathcal{X}\to X$ be a morphism from an algebraic superstack to an algebraic superspace. We say $\pi$ is a coarse moduli superspace if
    \begin{enumerate}[topsep=0pt,noitemsep,label=\normalfont(\arabic*)]
        \item\label{def:super_cms1} $\mathcal{O}_X\cong\pi_{*}\mathcal{O}_\mathcal{X}$,
        \item\label{def:super_cms2} $\pi$ is a universal homeomorphism,
        \item\label{def:super_cms3} $\Delta_{\pi}$ is universally submersive, and
        \item\label{def:super_cms4} $\pi$ is initial among morphisms to algebraic superspaces.
    \end{enumerate}
\end{definition}

If $\mathcal{X}$ is an algebraic superspace, for example, then the associated coarse moduli superspace is the identity morphism $\mathcal{X}\to\mathcal{X}$. Let $\mathcal{X}$ be an algebraic superstack with a flat presentation $U\to\mathcal{X}$. Let $R=U\times_\mathcal{X}U$. A morphism $\mathcal{X}\to X$ to an algebraic superspace is precisely a morphism $U\to X$ of algebraic superspaces that is equivariant under the action of $R$. Therefore, $\mathcal{X}\to X$ is a coarse moduli superspace if and only if the composition $U\to\mathcal{X}\to X$ is a strongly geometric quotient that is also categorical. We can now translate our results for supergroupoid quotients in the language of superstacks.  

\begin{theorem}\label{thm:superstacky_AF_mod_finite}
    Let $S$ be a superscheme. Let $\mathcal{X}$ be an algebraic superstack over $S$. If $\mathcal{X}$ admits a finite \'etale presentation $U\to\mathcal{X}$ where $U$ is an AF superscheme, then $\mathcal{X}$ admits a coarse moduli superspace $\pi\colon\mathcal{X}\to X$ such that $\pi$ is separated and the composition $U\to\mathcal{X}\to X$ is integral. If $\mathcal{X}$ is locally of finite type over $S$, then $\pi$ is proper and quasi-finite. If $S$ is also locally Noetherian, then $X$ is also locally of finite type over $S$.
\end{theorem}

\begin{proof}
    By \Cref{thm:quotients_by_finite_flat_supergroupoid}, we get a strongly geometric quotient $U\to X$ that is categorical. We need to check that this descends to a morphism $\pi\colon\mathcal{X}\to X$ satisfying the first three conditions in \Cref{def:super_cms}. \ref{def:super_cms1} follows from the fact that $U\to X$ is a geometric quotient. \ref{def:super_cms2} and \ref{def:super_cms3} follow from the fact that $U\to X$ is a strong topological quotient. To see that $\pi$ is separated, we have the following Cartesian diagram
    \begin{equation*}
        \begin{tikzcd}
            R=U\times_\mathcal{X}U\arrow[r]\arrow[d]\arrow[rd,phantom,"\square"]\arrow[rr,bend left=30] & U\times_XU\arrow[d]\arrow[r] & U\\
            \mathcal{X}\arrow[r,"\Delta_{\pi}"] & \mathcal{X}\times_X\mathcal{X}. &
        \end{tikzcd}
    \end{equation*}
    Note that $R\to U$ is finite and \'etale since $U\to\mathcal{X}$ is so. Moreover, the projection $U\times_X U\to U$ is separated. This is because the composition $U\to\mathcal{X}\to X$ is integral and thus separated. It follows that $R\to U\times_XU$ is finite. Since the diagram above is Cartesian, we see that $\Delta_\pi$ is finite and $\pi$ is separated. If $\mathcal{X}$ is locally of finite type over $S$, then $U\to X$ is finite by \Cref{thm:quotients_by_finite_flat_supergroupoid}. Therefore $\pi$ is quasi-finite and proper. If $S$ is also locally Noetherian, then $U$ and $X$ are also locally of finite type by \Cref{thm:quotients_by_finite_flat_supergroupoid}.
\end{proof}

Let $\pi\colon\mathcal{X}\to X$ be a coarse moduli superspace. We say that $\pi$ \textit{satisfies the descent condition for separated and \'etale morphisms} if for every representable, separated, \'etale, stabilizer-preserving morphism of superstacks $\mathcal{W}\to\mathcal{X}$, there exists a separated and \'etale morphism $W\to X$ of algebraic superspaces such that the diagram
    \begin{equation*}
        \begin{tikzcd}
            \mathcal{W}\arrow[r]\arrow[d]\arrow[rd, phantom,"\square"] & \mathcal{X}\arrow[d]\\
            W\arrow[r] & X
        \end{tikzcd}
    \end{equation*}
is Cartesian.

\begin{example}\label{ex:cmss_descent_conditions}
    Let $\pi\colon\mathcal{X}\to X$ be a coarse moduli superspace. If there exists a finite \'etale presentation $U\to\mathcal{X}$ such that $U$ is AF and the composition $U\to\mathcal{X}\to X$ is integral, then $\pi$ satisfies the descent condition for separated and \'etale morphisms. Indeed, we have a supergroupoid $(U,R)$ where $R=U\times_\mathcal{X}U$ such that the composition $U\to\mathcal{X}\to X$ is a strong geometric quotient. By \Cref{thm:super_geometric_quotient_descent_categorical}, we see that $U\to X$ satisfies the descent condition for separated and \'etale morphisms.

    Let $\mathcal{W}\to\mathcal{X}$ be a representable, separated, \'etale and stabilizer-preserving morphism of algebraic superstacks. Then there exists an induced morphism of supergroupoids $(R_\mathcal{W},U_\mathcal{W})\to(R,U)$ where $U_\mathcal{W}=U\times_\mathcal{X}\mathcal{W}$ and $R_\mathcal{W}=U_\mathcal{W}\times_\mathcal{W}U_\mathcal{W}$. Since $U\to X$ satisfies the descent condition for separated and \'etale morphisms, the diagram
    \begin{equation}\label{diag:descent_condition_def_2}
        \begin{tikzcd}
            U_\mathcal{W}\arrow[r]\arrow[d]\arrow[rd, phantom,"\square"] & U\arrow[d]\\
            W\arrow[r] & X
        \end{tikzcd}
    \end{equation}
    is Cartesian. Since $U_\mathcal{W}\cong U\times_\mathcal{X}\mathcal{W}$ and $R_\mathcal{W}\cong R\times_\mathcal{X}\mathcal{W}$, it follows that$$\mathcal{W}\cong[U_\mathcal{W}/R_\mathcal{W}]\cong[U/R]\times_{X}W$$as desired.
\end{example}

\begin{theorem}\label{thm:superstacky_Kollar_Rydh}
    Let $S$ be a superscheme. Let $f\colon\mathcal{W}\to\mathcal{X}$ be a representable, separated, \'etale, stabilizer-preserving, and surjective morphism of algebraic superstacks. Let $\mathcal{Q}=\mathcal{W}\times_\mathcal{X}\mathcal{W}$. Let $V\to\mathcal{W}$ be a finite \'etale presentation from a superscheme. If $\mathcal{W}$ admits a coarse moduli superspace that satisfies the descent condition for separated and \'etale morphisms, then the coarse moduli superspaces $\mathcal{Q}\to Q$ and $\mathcal{X}\to X$ exist such that the diagram
    \begin{equation}\label{dia:superstacky_Kollar_Rydh}
        \begin{tikzcd}
            \mathcal{Q}\arrow[r, shift left=1]\arrow[r, shift right=1]\arrow[d] & \mathcal{W}\arrow[r]\arrow[d] & \mathcal{X}\arrow[d]\\
            Q\arrow[r, shift left=1]\arrow[r, shift right=1] & W\arrow[r] & X
        \end{tikzcd}
    \end{equation}
    is Cartesian.
\end{theorem}

\begin{proof}
    We use the same gluing argument in the proof of \Cref{thm:super_Kollar_Rydh}. Apply the descent condition to the projections $\mathcal{Q}\rightrightarrows\mathcal{W}$. The resulting coarse superspaces form an \'etale supergroupoid $Q\rightrightarrows W$. Since the projections $\mathcal{Q}\rightrightarrows\mathcal{W}$ are stabilizer-preserving, we see that $Q\rightrightarrows W$ defines an \'etale equivalence relation. It follows that $W/Q$ is an algebraic superspace. Let $X=W/Q$. By \'etale descent, we have the desired Cartesian diagram in \ref{dia:superstacky_Kollar_Rydh}, and the induced maps $\mathcal{Q}\to Q$ and $\mathcal{X}\to X$ are coarse moduli superspaces.
\end{proof}

\begin{proof}[Proof of \Cref{thm:main}]
    Since $\mathcal{X}$ is Deligne--Mumford, it admits an \'etale presentation $U\to\mathcal{X}$ from a superscheme $U$. Since $\mathcal{X}$ has finite inertia, it has separated diagonal. Applying \Cref{lem:Nisnevich_covering_super} and \Cref{prop:stabilizer_preserving}, we get an \'etale, separated, surjective, representable, and stabilizer-preserving morphism $\mathcal{W}\to\mathcal{X}$ such that $\mathcal{W}$ admits a finite \'etale presentation $V\to\mathcal{W}$ where $V$ is an AF superscheme. By \Cref{thm:superstacky_AF_mod_finite}, there exists a coarse moduli superspace $\mathcal{W}\to W$ such that $V\to W$ is integral. In particular, $\mathcal{W}\to W$ satisfies the descent condition for separated and \'etale morphisms by \Cref{ex:cmss_descent_conditions}. Let $\mathcal{Q}=\mathcal{W}\times_\mathcal{X}\mathcal{W}$. Therefore by \Cref{thm:superstacky_Kollar_Rydh}, there exists a coarse moduli superspace $\pi\colon\mathcal{X}\to X$ such that the diagram
    \begin{equation}
        \begin{tikzcd}
            \mathcal{Q}\arrow[r, shift left=1]\arrow[r, shift right=1]\arrow[d] & \mathcal{W}\arrow[r]\arrow[d] & \mathcal{X}\arrow[d,"\pi"]\\
            Q\arrow[r, shift left=1]\arrow[r, shift right=1] & W\arrow[r] & X
        \end{tikzcd}
    \end{equation}
    is Cartesian and $\pi$ is separated.
    
    If $\mathcal{X}\to S$ is locally of finite type, then \Cref{thm:superstacky_AF_mod_finite} implies that $\mathcal{W}\to W$ is quasi-finite and proper. Since $\mathcal{W}\to\mathcal{X}$ is \'etale and surjective, and the right square above is Cartesian, we see that $\pi$ is quasi-finite and proper by descent. Quasi-compactness and universal openness and closedness follow from the fact that $\pi$ is a universal homeomorphism. Consider the following diagram
    \begin{equation*}
        \begin{tikzcd}
            \mathcal{X}\arrow[r,"\Delta_{\pi}",swap]\arrow[rr,"\Delta_{\mathcal{X}}",bend left=30] & \mathcal{X}\times_X\mathcal{X}\arrow[r,hook]\arrow[d] & \mathcal{X}\times_S\mathcal{X}\arrow[d,"\pi\times\pi"]\\
            & X\arrow[r,"\Delta_{X}",hook] & X\times_S X
        \end{tikzcd}
    \end{equation*}
    with a Cartesian square. Since $\pi$ is separated, $\Delta_\pi$ is proper. If $\mathcal{X}$ is quasi-separated (resp. separated), then the diagonal $\Delta_\mathcal{X}$ is quasi-compact (resp. proper).  So is the morphism $\mathcal{X}\times_X\mathcal{X}\to\mathcal{X}\times_S\mathcal{X}$. It follows that $\Delta_{X}$ is quasi-compact (resp. a closed immersion) since $\pi\times\pi$ is quasi-compact, universally closed and surjective.

    Now suppose $S$ is locally Noetherian. It suffices to show that if $\mathcal{X}\to S$ is locally of finite type, then $X\to S$ is also locally of finite type. By assumption, $\mathcal{W}\to S$ is locally of finite type. By \Cref{thm:superstacky_AF_mod_finite}, we see that $W\to S$ is locally of finite type. By construction, the map $W\to X$ is \'etale and surjective. It follows that $X\to S$ is locally of finite type. This finishes the proof.
\end{proof}

As remarked in the introduction, not every algebraic superstack with finite inertia admits a coarse moduli superspace. We end this section with the following counterexample.

\begin{example}\label{ex:counter_Keel_Mori}
    Consider the action in \Cref{ex:G_a^-action_not_integral}. Let $\mathcal{X}$ be the corresponding quotient superstack $[\mathbb{A}^{1|1}/\mathbb{G}_a^{0|1}]$ over $\mathbb{C}$. Since $\mathbb{A}^{1|1}$ is an affine superscheme and $\mathbb{G}_a^{0|1}$ is a finite flat group superscheme over $\mathbb{C}$, the quotient superstack $\mathcal{X}$ has finite diagonal, and thus finite inertia. We claim that $\mathcal{X}$ admits no coarse moduli superspace.
    
    Suppose that there exists a coarse moduli superspace $\pi\colon\mathcal{X}\to X$. We first claim that $X$ is an affine superscheme. Indeed, one checks that the composition$$\mathbb{A}^1\hookrightarrow\mathbb{A}^{1|1}\to\mathcal{X}\to X\to X_{\operatorname{ev}}$$is universally closed and surjective. Since $\mathbb{A}^{1}$ is affine and $X_{\operatorname{ev}}$ is an ordinary algebraic space (see \Cref{sec:AF}), we see that $X_{\operatorname{ev}}$ is affine by \cite[\href{https://stacks.math.columbia.edu/tag/07VT}{Tag 07VT}]{stacks-project}. It follows that $X$ is affine since the map $X\to X_{\operatorname{ev}}$ is affine by \Cref{cor:ev_map_integral}.
    
    Let $Y=\operatorname{sSpec}(\mathbb{C}[x|\xi]^{\mathbb{G}_a^{0|1}})$. The natural quotient morphism $\mathbb{A}^{1|1}\to Y$ descends to a morphism $\phi\colon\mathcal{X}\to Y$. By the universal property, we get a morphism $\phi\colon X\to Y$. Since $\pi\colon\mathcal{X}\to X$ is a coarse moduli space, we see that $\pi_*\mathcal{O_X}\cong\mathcal{O}_X$. On the other hand, we know that $Y$ is the spectrum of the superring of invariants. It follows that $\phi_*\mathcal{O_X}\cong\mathcal{O}_Y$. Therefore, we have $\overline{\phi}_*\mathcal{O}_X\cong\mathcal{O}_Y$. Since $X$ is affine, we must have $X\cong Y$.

    It remains to show $\pi\colon\mathcal{X}\to Y$ is not a coarse moduli space. We first observe that the bosonic reduction of $\mathcal{X}$ is just $\mathbb{A}^1$ over $\mathbb{C}$. However, the bosonic reduction of $Y$ is $\operatorname{Spec}(\mathbb{C})$. This is because the superring of invariants of this action is $\mathbb{C}\oplus\xi\mathbb{C}[x]$ as computed in \Cref{ex:G_a^-action_not_integral}. This tells us that $\pi$ cannot be a homeomorphism. Therefore $\mathcal{X}$ admits no coarse moduli spaces.
\end{example}

\appendix

\section{Bosonic quotients and the AF condition}\label{sec:AF}

In this appendix, we discuss the AF condition for algebraic superspaces and superstacks. We first explain how to construct the bosonic quotient of a Deligne--Mumford superstack. For complex Deligne--Mumford superstacks, this is \cite[Theorem\ 3.13(1)]{CV19}. The argument presented in \textit{loc.\@ cit.\@} applies to an arbitrary Deligne--Mumford superstacks.

Let $X$ be an algebraic superspace over $\mathbb{Z}[1/2]$. Choose an \'etale presentation $U\to X$ where $U$ is a superscheme. Let $R=U\times_XU$ and $s,t\colon R\to U$ be the natural projections. Then we have an \'etale supergroupoid $s,t\colon R\rightrightarrows U$. Taking the bosonic quotient of $s,t\colon R\rightrightarrows U$ gives $s_{\operatorname{ev}},t_{\operatorname{ev}}\colon R_{\operatorname{ev}}\rightrightarrows U_{\operatorname{ev}}$. We claim that this gives us an \'etale equivalence relation of ordinary schemes. Indeed, we first see that $s_{\operatorname{ev}},t_{\operatorname{ev}}$ are both \'etale since $s,t$ are so. By \cite[Lemma\ 7.10]{MZ24}, we have$$(R\times_{s,U,t}R)_{\operatorname{ev}}\cong R_{\operatorname{ev}}\times_{s_{\operatorname{ev}},U_{\operatorname{ev}},t_{\operatorname{ev}}}R_{\operatorname{ev}}.$$Therefore the unit, inverse, and composition maps descend to $s_{\operatorname{ev}},t_{\operatorname{ev}}\colon R_{\operatorname{ev}}\rightrightarrows U_{\operatorname{ev}}$, making it an \'etale groupoid. We define the bosonic quotient of $X$ as $X_{\operatorname{ev}}=[U_{\operatorname{ev}}/R_{\operatorname{ev}}]$. We claim that $\mathcal{X}_{\operatorname{ev}}$ is an ordinary algebraic space. Indeed, the supergroupoid $s,t\colon R\rightrightarrows U$ is an \'etale equivalence relation because $X$ is an algebraic superspace. It follows that the groupoid $s_{\operatorname{ev}},t_{\operatorname{ev}}\colon R_{\operatorname{ev}}\rightrightarrows U_{\operatorname{ev}}$ is an \'etale equivalence relation by the same geometric point argument in the proof of \Cref{thm:superspace_finite_flat_covering}. This tells us that the bosonic quotient of an algebraic superspace is an algebraic space. 

\begin{lemma}[{cf. \cite[Lemma\ 7.10]{MZ24}}]\label{lem:ev_commutes_etale_base_change}
    Let $S$ be an algebraic superspace. Let $X,Y$ be algebraic superspaces over $S$.  If $X\to S$ is \'etale, then so is the induced map $X_{\operatorname{ev}}\to S_{\operatorname{ev}}$ and the following diagram
    \begin{equation*}
        \begin{tikzcd}
            X\arrow[r]\arrow[d]\arrow[rd, phantom,"\square"] & S\arrow[d]\\
            X_{\operatorname{ev}}\arrow[r] & S_{\operatorname{ev}}
        \end{tikzcd}
    \end{equation*}
    is Cartesian. In particular, the canonical morphism$$(X\times_SY)_{\operatorname{ev}}\longrightarrow X_{\operatorname{ev}}\times_{S_{\operatorname{ev}}}Y_{\operatorname{ev}}$$is an isomorphism.
\end{lemma}

\begin{proof}
    The statement is \'etale-local both on the source and target. So we may assume that $X,Y$ and $S$ are affine superschemes. Since $X\to S$ is \'etale, it is formally \'etale and locally of finite presentation. The statement follows from the proof of \cite[Proposition A.26]{BHP23}.  
\end{proof}

Although \cite[Proposition A.26]{BHP23} is stated over a locally Noetherian base, its proof readily extends to \'etale morphisms over an arbitrary base. Indeed, the argument is local and the use of Nakayama's lemma only concerns finitely generated modules coming from a morphism locally of finite presentation. Therefore the canonical morphism $B_0\otimes_{A_0}A\to B$ is still an isomorphism in the proof of \cite[Proposition A.26]{BHP23}, and the rest of the argument shows that $A_0\to B_0$ is \'etale. We record the following fact about the bosonic quotient map as a consequence of \Cref{lem:ev_commutes_etale_base_change}.

\begin{corollary}\label{cor:ev_map_integral}
    Let $X$ be an algebraic superspace. Then the canonical bosonic quotient map $X\to X_{\operatorname{ev}}$ is integral.
\end{corollary}

\begin{proof}
    Choose an \'etale presentation $U\to X$ from a superscheme $U$. We may assume $U$ is a disjoint union of affine superscheme by choosing a suitable open cover. By \Cref{lem:ev_commutes_etale_base_change} and \Cref{lem:integral_morphism_etale_local}, we may assume that $U=X$. In this case, the statement follows from \Cref{lem:integral_bosonic}.
\end{proof}

We may then use \Cref{lem:ev_commutes_etale_base_change} to define the bosonic quotient of a Deligne--Mumford superstack over $\mathbb{Z}[1/2]$. Let $\mathcal{X}$ be a Deligne--Mumford superstack over $\mathbb{Z}[1/2]$. Choose an \'etale presentation $U\to \mathcal{X}$ where $U$ is a superscheme. Let $R=U\times_\mathcal{X}U$ and $s,t\colon R\to U$ be the natural projections. Note that $R$ is an algebraic superspace. Then we have an \'etale supergroupoid $s,t\colon R\rightrightarrows U$. Taking the bosonic quotient of $s,t\colon R\rightrightarrows U$ gives $s_{\operatorname{ev}},t_{\operatorname{ev}}\colon R_{\operatorname{ev}}\rightrightarrows U_{\operatorname{ev}}$. We claim that this gives us an \'etale groupoid relation of ordinary algebraic spaces. Indeed, we first see that $s_{\operatorname{ev}},t_{\operatorname{ev}}$ are both \'etale since $s,t$ are so. By \Cref{lem:ev_commutes_etale_base_change}, we have$$(R\times_{s,U,t}R)_{\operatorname{ev}}\cong R_{\operatorname{ev}}\times_{s_{\operatorname{ev}},U_{\operatorname{ev}},t_{\operatorname{ev}}}R_{\operatorname{ev}}.$$Therefore the unit, inverse, and composition maps descend to $s_{\operatorname{ev}},t_{\operatorname{ev}}\colon R_{\operatorname{ev}}\rightrightarrows U_{\operatorname{ev}}$, making it an \'etale groupoid. It follows that the quotient $[U_{\operatorname{ev}}/R_{\operatorname{ev}}]$ is an ordinary Deligne--Mumford stack. We define the bosonic quotient of $\mathcal{X}$ as $\mathcal{X}_{\operatorname{ev}}=[U_{\operatorname{ev}}/R_{\operatorname{ev}}]$. Moreover, there is a natural quotient map $\mathcal{X}\to\mathcal{X}_{\operatorname{ev}}$. To see this, we observe that the composition $U\to U_{\operatorname{ev}}\to\mathcal{X}_{\operatorname{ev}}$ is $R$-invariant. Therefore it descends to a morphism $\mathcal{X}\to\mathcal{X}_{\operatorname{ev}}$ as desired. 

Let $Y$ be an ordinary Deligne--Mumford stack. We claim that there is an equivalence$$\operatorname{Hom}_{\operatorname{DM}}(\mathcal{X}_{\operatorname{ev}},Y)\simeq\operatorname{Hom}_{\operatorname{sDM}}(\mathcal{X},Y).$$To see this, choose an \'etale presentation $U\to X$ and let $R=U\times_\mathcal{X}U$. Consider the supergroupoid $(U,R)$. We know that a morphism from $\mathcal{X}\to\mathcal{Y}$ is the same as an $R$-invariant morphism $U\to\mathcal{Y}$ from a superscheme satisfying the cocycle condition. But we have$$\operatorname{Hom}_{\operatorname{DM}}(U_{\operatorname{ev}},Y)\simeq\operatorname{Hom}_{\operatorname{sDM}}(U,Y)$$and the claim follows. This implies that taking the bosonic quotient is the left adjoint of the natural inclusion from the 2-category of Deligne--Mumford stacks to that of Deligne--Mumford superstacks.

Recall that a superscheme is AF if every finite set of points is contained in some affine open subsuperscheme. We say a morphism of algebraic superspaces $X\to Y$ is AF if for every morphism $U\to Y$ from an affine superscheme $U$, the fibre product $X\times_YU$ is an AF superscheme. For example, an affine morphism is AF.
 
\begin{proposition}\label{prop:superscheme_AF}
    Let $f\colon X\to S$ be a separated and \'etale morphism of algebraic superspaces. If $S$ is AF, so is $X$.
\end{proposition}

\begin{proof}
    By \Cref{lem:ev_commutes_etale_base_change}, we have the following Cartesian diagram
    \begin{equation*}
        \begin{tikzcd}
            X\arrow[r,"f"]\arrow[d] & S\arrow[d]\\
            X_{\operatorname{ev}}\arrow[r,"f_{\operatorname{ev}}"] & S_{\operatorname{ev}}
        \end{tikzcd}
    \end{equation*}
    where the square is Cartesian. By \cite[Proposition B.2(vi) and (iv)]{Ryd13}, the map $f_{\operatorname{ev}}$ and the space $X_{\operatorname{ev}}$ are both AF. Since $X\to X_{\operatorname{ev}}$ is affine by \Cref{cor:ev_map_integral}, we see that $X$ is also AF.
\end{proof}

\section{The super Hilbert space of points}\label{sec:SHilb}

In this appendix, we establish the representability of the relative super Hilbert functor of separated, \'etale, and representable morphisms of algebraic superstacks. This was used in the proof of \Cref{lem:Nisnevich_covering_super}. We do so via the relative Hilbert super functor of points. In the ordinary case, this is due to \cite{Ryd11a}. In the super case, it was first considered by \cite{Jan20}. Let $X\to S$ be a separated algebraic superspace. We first recall the definition of the super Hilbert functor of points.

\begin{definition}\label{def:finite_flat_super}
    Let $f\colon X\to Y$ be a finite morphism of algebraic superspaces. We say $f$ is \textit{flat of rank $p|q$} if $f_{*}\mathcal{O}_X$ is a locally free sheaf on $Y$ of rank $p|q$.
\end{definition}

\begin{definition}\label{def:super_hilbert_functor}
    Let $X\to S$ be a separated algebraic superstack. The \textit{super Hilbert functor of points} $\mathcal{SH}\textit{ilb}_{X/S}$ is the functor that assigns, to a superscheme $T\to S$, the set of closed subsuperspaces $Z\hookrightarrow X\times_ST$ such that the composition $Z\hookrightarrow X\times_ST\to T$ is finite flat and locally of finite presentation. The \textit{super Hilbert functor of points} $\mathcal{SH}\textit{ilb}_{X/S}^{p|q}$ is the functor that assigns, to a superscheme $T\to S$, the set of closed subsuperspaces $Z\hookrightarrow X\times_ST$ such that the composition $Z\hookrightarrow X\times_ST\to T$ is finite flat of rank $p|q$. 
\end{definition}

It is not hard to see from the definition that $\mathcal{SH}\textit{ilb}_{X/S}^{p|q}$ is a sheaf on $\mathbf{sSch}_{/S}$ for the fpqc topology. Let $f\colon X\to Y$ be a morphism of algebraic superspaces over $S$. If $f$ is a closed immersion, then there is a natural transformation $f_{*}\colon\mathcal{SH}\textit{ilb}_{X/S}^{p|q}\to\mathcal{SH}\textit{ilb}_{Y/S}^{p|q}$, which sends a closed subscheme $Z\hookrightarrow X\times_ST$ to $f_{T}(Z)\hookrightarrow Y\times_ST$ where $f_{T}\colon X\times_ST\to Y\times_ST$ for every superscheme $T\to S$. In this appendix, we will focus on the functor $\mathcal{SH}\textit{ilb}_{X/S}^{d|0}$ where $X\to S$ is separated and \'etale. In this case, we have $\mathcal{SH}\textit{ilb}_{X/S}=\coprod_{d\geq 0}\mathcal{SH}\textit{ilb}_{X/S}^{d|0}$.

\begin{proposition}\label{prop:SHilb_quasi_finite}
    Let $f\colon X\to S$ be a separated and \'etale morphism of algebraic superspaces. Then the relative super Hilbert functor $\mathcal{SH}\textit{ilb}_{X/S}$ is representable by an algebraic superspace over $S$ such that $\mathcal{SH}\textit{ilb}_{X/S}\cong\mathcal{H}\textit{ilb}_{X_{\operatorname{ev}}/S_{\operatorname{ev}}}\times_{S_{\operatorname{ev}}}S$. In particular, $\mathcal{SH}\textit{ilb}_{X/S}$ is separated and \'etale over $S$.
\end{proposition}

\begin{proof}
    It suffices to prove the statement for $\mathcal{SH}\textit{ilb}_{X/S}^{d|0}$ for every $d\geq 0$. Let $T\to S$ be a morphism of superschemes. An object in $\mathcal{SH}\textit{ilb}_{X/S}^{d|0}(T)$ is a closed immersion $Z\hookrightarrow X\times_ST\to T$ that is finite and \'etale of rank $d|0$. Indeed, the composition $Z\to T$ is finite locally free and unramified and thus \'etale. In particular, $Z\hookrightarrow X\times_ST$ is an open and closed immersion. Since $f\colon X\to S$ is \'etale, we have $(X\times_ST)_{\operatorname{ev}}\cong X_{\operatorname{ev}}\times_{S_{\operatorname{ev}}}T_{\operatorname{ev}}$. Therefore taking bosonic quotients yields a family $Z_{\operatorname{ev}}\hookrightarrow X_{\operatorname{ev}}\times_{S_{\operatorname{ev}}}T_{\operatorname{ev}}\to T_{\operatorname{ev}}$ that is finite and \'etale of degree $d$ over $T_{\operatorname{ev}}$. This is precisely an object in $(\mathcal{H}\textit{ilb}_{X_{\operatorname{ev}}/S_{\operatorname{ev}}}^d\times_{S_{\operatorname{ev}}}S)(T)$. On the other hand, we see that $(\mathcal{H}\textit{ilb}_{X_{\operatorname{ev}}/S_{\operatorname{ev}}}^d\times_{S_{\operatorname{ev}}}S)(T)=\mathcal{H}\textit{ilb}_{X_{\operatorname{ev}}/S_{\operatorname{ev}}}^d(T_{\operatorname{ev}})$. Taking the base change along $T\to T_{\operatorname{ev}}$ sends an object in $(\mathcal{H}\textit{ilb}_{X_{\operatorname{ev}}/S_{\operatorname{ev}}}^d\times_{S_{\operatorname{ev}}}S)(T)$ to an object in $\mathcal{SH}\textit{ilb}_{X/S}^{d|0}(T)$. One readily checks that these two constructions are quasi-inverses of each other. Since this is true for an arbitrary superscheme $T\to S$, the result follows. The last statement then follows from the fact that $\mathcal{H}\textit{ilb}_{X_{\operatorname{ev}}/S_{\operatorname{ev}}}^d$ is \'etale over $S_{\operatorname{ev}}$ by \cite[Theorem 5.1]{Ryd11a}. 

    For separatedness, we argue as in \cite[Lemma 2.1]{Ryd11a}. For every superscheme $T$ over $S$, a morphism $T\to \mathcal{SH}\textit{ilb}_{X/S}^{d|0}\times_S\mathcal{SH}\textit{ilb}_{X/S}^{d|0}$ corresponds to a pair of closed subschemes $Z_{1},Z_{2}\hookrightarrow X\times_ST$. Define the functor
    \begin{align*}
        T_\Delta\colon(\mathbf{sSch}_{/T})^{op}&\longrightarrow\mathbf{Set}\\
        (T^\prime\to T)&\longmapsto\begin{cases}
            \{*\}\ \ \text{if}\ Z_{1}\times_{T}T^\prime=Z_{2}\times_{T}T^\prime\\
            \varnothing\ \ \text{otherwise}.
        \end{cases}
    \end{align*}
    This functor is representable by a closed subsuperscheme of $T$. Therefore, we have the following Cartesian square
    \begin{equation*}
    \begin{tikzcd}
        T_\Delta\arrow[r]\arrow[d]\arrow[dr,phantom,"\square"] & T\arrow[d]\\
        \mathcal{SH}\textit{ilb}_{X/S}^{d|0}\arrow[r,"\Delta"] & \mathcal{SH}\textit{ilb}_{X/S}^{d|0}\times_S\mathcal{SH}\textit{ilb}_{X/S}^{d|0}
    \end{tikzcd}
\end{equation*}
where the top horizontal map is a closed immersion. It follows that the diagonal of $\mathcal{SH}\textit{ilb}_{X/S}^{d|0}$ is a closed immersion. This finishes the proof.
\end{proof}

\begin{corollary}\label{cor:SHilb_quasi_finite}
    Let $f\colon X\to S$ be a separated \'etale and representable morphism of algebraic superstacks. Then the relative super Hilbert functor $\mathcal{SH}\textit{ilb}_{X/S}\to S$ is separated, \'etale and representable by an algebraic superspace.
\end{corollary}

\begin{proof}
    The statement is local on $S$. So we may assume that $S$ is an algebraic superspace. The statement then follows from \Cref{prop:SHilb_quasi_finite}.
\end{proof}

\bibliographystyle{alpha}
\bibliography{super.bib}

@article{Westra:2009,
  doi = {10.25365/THESIS.6869},
  url = {https://utheses.univie.ac.at/detail/6203},
  author = {Westra,  Dennis},
  language = {en},
  title = {Superrings and supergroups},
  publisher = {Universit{¨a}t Wien, Wien},
  year = {2009}
}

@ARTICLE{RTT23,
       author = {{Rizzo}, Pedro and {Torres Del Valle}, Joel and {Torres-Gomez}, Alexander},
        title = "{Dedekind Superrings and Related Concepts}",
      journal = {arXiv e-prints},
         year = 2023,
        month = oct,
          eid = {arXiv:2310.03822},
        pages = {arXiv:2310.03822},
          doi = {10.48550/arXiv.2310.03822},
archivePrefix = {arXiv},
       eprint = {2310.03822},
 primaryClass = {math.RA},
       adsurl = {https://ui.adsabs.harvard.edu/abs/2023arXiv231003822R}
}

@misc{stacks-project,
  author       = {The {Stacks project authors}},
  title        = {The Stacks project},
  howpublished = {\url{https://stacks.math.columbia.edu}},
  year         = {2025},
}

@article {BHP23,
    AUTHOR = {Bruzzo, Ugo and Hern\'andez Ruip\'erez, Daniel and Polishchuk,
              Alexander},
     TITLE = {Notes on fundamental algebraic supergeometry. {H}ilbert and
              {P}icard superschemes},
   JOURNAL = {Adv. Math.},
  FJOURNAL = {Advances in Mathematics},
    VOLUME = {415},
      YEAR = {2023},
     PAGES = {Paper No. 108890, 115},
      ISSN = {0001-8708,1090-2082},
   MRCLASS = {14M30 (14D22 14H10 14K10 83E30)},
  MRNUMBER = {4544562},
MRREVIEWER = {Frans\ Cantrijn},
       DOI = {10.1016/j.aim.2023.108890},
       URL = {https://doi.org/10.1016/j.aim.2023.108890},
}

@article {CV19,
    AUTHOR = {Codogni, Giulio and Viviani, Filippo},
     TITLE = {Moduli and periods of supersymmetric curves},
   JOURNAL = {Adv. Theor. Math. Phys.},
  FJOURNAL = {Advances in Theoretical and Mathematical Physics},
    VOLUME = {23},
      YEAR = {2019},
    NUMBER = {2},
     PAGES = {345--402},
      ISSN = {1095-0761,1095-0753},
   MRCLASS = {14M30 (14H10 81T60)},
  MRNUMBER = {4033354},
MRREVIEWER = {Anargyros\ Fellouris},
       DOI = {10.4310/ATMP.2019.v23.n2.a2},
       URL = {https://doi.org/10.4310/ATMP.2019.v23.n2.a2},
}

@book {Ols16,
    AUTHOR = {Olsson, Martin},
     TITLE = {Algebraic spaces and stacks},
    SERIES = {American Mathematical Society Colloquium Publications},
    VOLUME = {62},
 PUBLISHER = {American Mathematical Society, Providence, RI},
      YEAR = {2016},
     PAGES = {xi+298},
      ISBN = {978-1-4704-2798-6},
   MRCLASS = {14D23 (14D22)},
  MRNUMBER = {3495343},
MRREVIEWER = {Stefan\ Schr\"oer},
       DOI = {10.1090/coll/062},
       URL = {https://doi.org/10.1090/coll/062},
}

@article {ATT72,
    AUTHOR = {Atterton, Thomas William},
     TITLE = {Definitions of integral elements and quotient rings over
              non-commutative rings with identity},
   JOURNAL = {J. Austral. Math. Soc.},
  FJOURNAL = {J. Austral. Math. Soc.},
    VOLUME = {13},
      YEAR = {1972},
     PAGES = {433--446},
   MRCLASS = {16A08},
  MRNUMBER = {313288},
MRREVIEWER = {R.\ A.\ Beauregard},
}

@ARTICLE{Con05,
    AUTHOR = "Conrad, Brian",
    TITLE = "Keel-Mori theorem via stacks",
    URL = "https://math.stanford.edu/\textasciitilde conrad/papers/coarsespace.pdf",
    YEAR = "2005"
}

@article {Ryd10,
    AUTHOR = {Rydh, David},
     TITLE = {Submersions and effective descent of \'etale morphisms},
   JOURNAL = {Bull. Soc. Math. France},
  FJOURNAL = {Bulletin de la Soci\'et\'e{} Math\'ematique de France},
    VOLUME = {138},
      YEAR = {2010},
    NUMBER = {2},
     PAGES = {181--230},
      ISSN = {0037-9484,2102-622X},
   MRCLASS = {14A15 (13B22 13B40 14F20 14F43)},
  MRNUMBER = {2679038},
MRREVIEWER = {Liam\ O'Carroll},
       DOI = {10.24033/bsmf.2588},
       URL = {https://doi.org/10.24033/bsmf.2588},
}

@ARTICLE{BH25,
       author = {{Bruzzo}, Ugo and {Hern{\'a}ndez Ruip{\'e}rez}, Daniel},
        title = "{Foundations of superstack theory}",
      journal = {arXiv e-prints},
         year = 2025,
        month = may,
          eid = {arXiv:2505.19899},
        pages = {arXiv:2505.19899},
          doi = {10.48550/arXiv.2505.19899},
archivePrefix = {arXiv},
       eprint = {2505.19899},
 primaryClass = {math.AG},
       adsurl = {https://ui.adsabs.harvard.edu/abs/2025arXiv250519899B}
}

@article {MZ24,
    AUTHOR = {Moosavian, Seyed Faroogh and Zhou, Yehao},
     TITLE = {On the existence of heterotic-string and type-{II}-superstring
              field theory vertices},
   JOURNAL = {J. Geom. Phys.},
  FJOURNAL = {Journal of Geometry and Physics},
    VOLUME = {205},
      YEAR = {2024},
     PAGES = {Paper No. 105307, 73},
      ISSN = {0393-0440,1879-1662},
   MRCLASS = {81T30 (14H81)},
  MRNUMBER = {4794509},
       DOI = {10.1016/j.geomphys.2024.105307},
       URL = {https://doi.org/10.1016/j.geomphys.2024.105307},
}

@article {Ryd13,
    AUTHOR = {Rydh, David},
     TITLE = {Existence and properties of geometric quotients},
   JOURNAL = {J. Algebraic Geom.},
  FJOURNAL = {Journal of Algebraic Geometry},
    VOLUME = {22},
      YEAR = {2013},
    NUMBER = {4},
     PAGES = {629--669},
      ISSN = {1056-3911,1534-7486},
   MRCLASS = {14D23 (14A20 18B40)},
  MRNUMBER = {3084720},
MRREVIEWER = {Stefan\ Schr\"oer},
       DOI = {10.1090/S1056-3911-2013-00615-3},
       URL = {https://doi.org/10.1090/S1056-3911-2013-00615-3},
}

@article {LR88,
    AUTHOR = {LeBrun, Claude and Rothstein, Mitchell},
     TITLE = {Moduli of super {R}iemann surfaces},
   JOURNAL = {Comm. Math. Phys.},
  FJOURNAL = {Communications in Mathematical Physics},
    VOLUME = {117},
      YEAR = {1988},
    NUMBER = {1},
     PAGES = {159--176},
      ISSN = {0010-3616,1432-0916},
   MRCLASS = {32G15 (14H15 81E30)},
  MRNUMBER = {946998},
MRREVIEWER = {V.\ S.\ Retakh},
       URL = {http://projecteuclid.org/euclid.cmp/1104161598},
}

@article {DHS97,
    AUTHOR = {Dom\'inguez P\'erez, J. A. and Hern\'andez Ruip\'erez, D. and
              Sancho de Salas, C.},
     TITLE = {Global structures for the moduli of (punctured) super
              {R}iemann surfaces},
   JOURNAL = {J. Geom. Phys.},
  FJOURNAL = {Journal of Geometry and Physics},
    VOLUME = {21},
      YEAR = {1997},
    NUMBER = {3},
     PAGES = {199--217},
      ISSN = {0393-0440,1879-1662},
   MRCLASS = {14M30 (14H10 32C11)},
  MRNUMBER = {1429097},
MRREVIEWER = {Sergey\ Merkulov},
       DOI = {10.1016/S0393-0440(96)00016-2},
       URL = {https://doi.org/10.1016/S0393-0440(96)00016-2},
}

@ARTICLE{BH25b,
       author = {{Bruzzo}, Ugo and {Hern{\'a}ndez Ruip{\'e}rez}, Daniel},
        title = "{Moduli of stable supermaps}",
      journal = {arXiv e-prints},
         year = 2025,
        month = may,
          eid = {arXiv:2505.22233},
        pages = {arXiv:2505.22233},
          doi = {10.48550/arXiv.2505.22233},
archivePrefix = {arXiv},
       eprint = {2505.22233},
 primaryClass = {math.AG},
       adsurl = {https://ui.adsabs.harvard.edu/abs/2025arXiv250522233B}
}

@article {KM97,
    AUTHOR = {Keel, Se\'an and Mori, Shigefumi},
     TITLE = {Quotients by groupoids},
   JOURNAL = {Ann. of Math. (2)},
  FJOURNAL = {Annals of Mathematics. Second Series},
    VOLUME = {145},
      YEAR = {1997},
    NUMBER = {1},
     PAGES = {193--213},
      ISSN = {0003-486X,1939-8980},
   MRCLASS = {14D25 (14L30)},
  MRNUMBER = {1432041},
MRREVIEWER = {Andrzej\ Bia\l ynicki-Birula},
       DOI = {10.2307/2951828},
       URL = {https://doi.org/10.2307/2951828},
}

@ARTICLE{DO23,
       author = {{Donagi}, Ron and {Ott}, Nadia},
        title = "{Supermoduli Space with Ramond punctures is not projected}",
      journal = {arXiv e-prints},
         year = 2023,
        month = aug,
          eid = {arXiv:2308.07957},
        pages = {arXiv:2308.07957},
          doi = {10.48550/arXiv.2308.07957},
archivePrefix = {arXiv},
       eprint = {2308.07957},
 primaryClass = {math.AG},
       adsurl = {https://ui.adsabs.harvard.edu/abs/2023arXiv230807957D}
}

@article {FKP23,
    AUTHOR = {Felder, Giovanni and Kazhdan, David and Polishchuk, Alexander},
     TITLE = {The moduli space of stable supercurves and its canonical line
              bundle},
   JOURNAL = {Amer. J. Math.},
  FJOURNAL = {American Journal of Mathematics},
    VOLUME = {145},
      YEAR = {2023},
    NUMBER = {6},
     PAGES = {1777--1886},
      ISSN = {0002-9327,1080-6377},
   MRCLASS = {14D23 (14H10 14M30)},
  MRNUMBER = {4684291},
MRREVIEWER = {Anargyros\ Fellouris},
       DOI = {10.1353/ajm.2023.a913296},
       URL = {https://doi.org/10.1353/ajm.2023.a913296},
}

@article {Wit19,
    AUTHOR = {Witten, Edward},
     TITLE = {Notes on super {R}iemann surfaces and their moduli},
   JOURNAL = {Pure Appl. Math. Q.},
  FJOURNAL = {Pure and Applied Mathematics Quarterly},
    VOLUME = {15},
      YEAR = {2019},
    NUMBER = {1},
     PAGES = {57--211},
      ISSN = {1558-8599,1558-8602},
   MRCLASS = {32C11 (14M30 58A50 81T30)},
  MRNUMBER = {3946083},
MRREVIEWER = {Kowshik\ Bettadapura},
       DOI = {10.4310/PAMQ.2019.v15.n1.a2},
       URL = {https://doi.org/10.4310/PAMQ.2019.v15.n1.a2},
}

@article {Ryd11a,
    AUTHOR = {Rydh, David},
     TITLE = {Representability of {H}ilbert schemes and {H}ilbert stacks of
              points},
   JOURNAL = {Comm. Algebra},
  FJOURNAL = {Communications in Algebra},
    VOLUME = {39},
      YEAR = {2011},
    NUMBER = {7},
     PAGES = {2632--2646},
      ISSN = {0092-7872,1532-4125},
   MRCLASS = {14C05 (14A20 14D23)},
  MRNUMBER = {2821738},
MRREVIEWER = {Li\ Li},
       DOI = {10.1080/00927872.2010.488678},
       URL = {https://doi.org/10.1080/00927872.2010.488678},
}

@article {Jan20,
    AUTHOR = {Jang, Mi Young},
     TITLE = {Families of 0-dimensional subspaces on supercurves of
              dimension {$1\mid1$}},
   JOURNAL = {J. Pure Appl. Algebra},
  FJOURNAL = {Journal of Pure and Applied Algebra},
    VOLUME = {224},
      YEAR = {2020},
    NUMBER = {6},
     PAGES = {106251, 16},
      ISSN = {0022-4049,1873-1376},
   MRCLASS = {14M30 (32C11)},
  MRNUMBER = {4048511},
MRREVIEWER = {Frans\ Cantrijn},
       DOI = {10.1016/j.jpaa.2019.106251},
       URL = {https://doi.org/10.1016/j.jpaa.2019.106251},
}

@article {Nor26,
    AUTHOR = {Norbury, Paul},
     TITLE = {Enumerative geometry via the moduli space of super {R}iemann
              surfaces},
   JOURNAL = {J. Geom. Phys.},
  FJOURNAL = {Journal of Geometry and Physics},
    VOLUME = {222},
      YEAR = {2026},
     PAGES = {Paper No. 105750, 57},
      ISSN = {0393-0440,1879-1662},
   MRCLASS = {32G15 (14H81 14M35 32C11 58A50)},
  MRNUMBER = {5018360},
       DOI = {10.1016/j.geomphys.2025.105750},
       URL = {https://doi.org/10.1016/j.geomphys.2025.105750},
}

@ARTICLE{Dan26,
       author = {{Dang}, Marcel},
        title = "{Derived Methods in Supergeometry: Fundamental Classes and Cohomology}",
      journal = {arXiv e-prints},
         year = 2026,
        month = jun,
          eid = {arXiv:2606.09897},
        pages = {arXiv:2606.09897},
          doi = {10.48550/arXiv.2606.09897},
archivePrefix = {arXiv},
       eprint = {2606.09897},
 primaryClass = {math.AG},
       adsurl = {https://ui.adsabs.harvard.edu/abs/2026arXiv260609897D}
}

@unpublished{OPP26,
  author    = {Ott, Nadia and Polishchuk, Alexander and Peng, Fei},
  title     = {},
  note      = {In preparation},
  year      = {2026}
}

\end{document}